\documentclass[12pt]{article}
\usepackage{amssymb,amsfonts,amsmath,amsthm,amsrefs,dsfont,hyperref}
\usepackage[letterpaper, portrait, margin=25.5mm]{geometry}

\newcommand{\1}{\mathds{1}}
\newcommand{\0}{\mathds{O}}

\newcommand{\Q}{\mathbb{Q}}
\newcommand{\R}{\mathbb{R}}

\newcommand{\N}{\mathbb{N}}

\newcommand{\8}{\infty}

\newcommand{\supp}{\mathrm{supp}}
\newcommand{\Co}{\mathcal{C}}

\newcommand{\Int}{\mathrm{int}}

\newcounter{dummy} \numberwithin{dummy}{section}
\newtheorem{theorem}[dummy]{Theorem}
\newtheorem{lemma}[dummy]{Lemma}

\newtheorem{proposition}[dummy]{Proposition}
\newtheorem{corollary}[dummy]{Corollary}
\newtheorem{question}[dummy]{Question}
\theoremstyle{remark}
\newtheorem{remark}[dummy]{Remark}
\newtheorem{example}[dummy]{Example}

\begin{document}

\title{Locally solid convergences and order continuity of positive operators}
\author{Eugene Bilokopytov\footnote{University of Alberta; email address bilokopy@ualberta.ca, erz888@gmail.com.}}
\maketitle

\begin{abstract}
We consider vector lattices endowed with locally solid convergence structures, which are not necessarily topological. We show that such a convergence is defined by the convergence to $0$ on the positive cone. Some results on unbounded modification which were only available in partial cases are generalized. Order convergence is characterized as the strongest locally solid convergence in which monotone nets converge to their extremums (if they exist). We partially characterize sublattices on which the order convergence is the restriction of the order convergence on the ambient lattice. We prove that a homomorphism is order continuous iff it is uo-continuous. Uo convergence is characterized independently of order convergence. We show that on the space of continuous function uo convergence is weaker than the compact open convergence iff the underlying topological space contains a dense locally compact subspace. For a large class of convergences we prove that a positive operator is order continuous if and only if its restriction to a dense regular sublattice is order continuous, and that the closure of a regular sublattice is regular with the original sublattice being order dense in the closure. We also present an example of a regular sublattice of a locally solid topological vector lattice whose closure is not regular.\\

\emph{Keywords:} Vector lattices, Regular sublattices, Order continuity;

MSC2020 46A19, 46A40, 47B60.
\end{abstract}

\section{Introduction}

In this article we study various issues related to order continuity. We approach the problem from the point of view of the theory of (net) convergence, which is a generalization of topology that has proved to be appropriate for studying vector lattices. Indeed, there are several convergences that are induced by the order and algebraic structures of a vector lattice, such as the order convergence and the uniform convergence, which fail to be topological in most cases (see \cite{dem}). Moreover, these convergences are defined in terms of sequences or nets, and so the well developed language of the convergence of filters (see \cite{bb}) is not very convenient for the applications. As a consequence, very recently there has been an effort to develop a general approach to net convergence in the context of vector lattices (see \cite{aeg}, \cite{dow}, \cite{vw}), and this work aims to be a contribution to the project. Let us now describe the content of the article.\medskip

In Section \ref{lscs} we study locally solid convergences in the most general form. We work in the category of net convergence spaces, which was shown in \cite{dow} to be equivalent to the category of filter convergence spaces. The latter is significantly broader than the category of topological spaces, but still rich enough and well studied (see \cite{bb}).

We show (Theorem \ref{locs}) that in order to define a locally solid additive or linear convergence on a vector lattice, it is enough to specify which nets of positive elements converge to the origin, as long as this restricted convergence satisfies only three or four natural conditions. This result is crafted towards the way convergences on vector lattices are defined in practice, and should reduce the effort of the verification if a given convergence is ``admissible''.

In Section \ref{ums} we study the unbounded modifications of locally solid convergences. In particular we prove the general versions of the results, which were previously available either for the unbounded modifications of locally solid topologies or for the unbounded order convergence.\medskip

Starting with Section \ref{ocvls} the order convergence becomes central to our investigation. First, we formally establish (Theorem \ref{ordex}) the intuitive property of this convergence as the one which is determined by the monotone nets converging to their extremums. In Section \ref{ochs} we show (Theorem \ref{huo}) that a homomorphism is order continuous iff it is uo-continuous. We also investigate (Proposition \ref{oh2}) when the ``intrinsic'' order convergence on a sublattice is the restriction of the order convergence of the ambient vector lattice. In Section \ref{cuos} we resolve a somewhat paradoxical situation in which uo convergence often has better properties than the order convergence, but occupies a ``derivative'' role with respect to it. To remedy this we present (Theorem \ref{uo}) an alternative definition of uo convergence, which better captures the analogy to the almost everywhere convergence (note that the uo convergence for sequences in $L_{p}$ is the convergence almost everywhere, see \cite[Proposition 3.1]{gtx}).

Here we would like to remark that the content of sections \ref{ocvls}, \ref{ochs} and \ref{cuos} has a large intersection with the works \cite{pap1} and \cite{pap2} of Fredos Papangelou. In fact, we present some of his results, but with optimized proofs and modernized notations. Namely, Proposition \ref{ocs} shows that the notion of order convergence used in the theory of vector lattices is consistent with the notion of order convergence on partially ordered sets, Theorem \ref{gt1} characterizes regularity of a sublattice in terms of order and uo properties of the inclusion map, and Theorem \ref{uo2} gives yet another characterization of uo convergence. Note that Theorem \ref{gt1} was also rediscovered in \cite{as} and \cite{gtx}, but we present a modification of Papangelou's proof which is shorter and more direct.\medskip

In Section \ref{uocsfs} we use the new characterization of uo convergence to provide a quick proof of the criterion (Theorem \ref{couo1}) of uo convergence on the space of continuous functions recently obtained in \cite{et}. We also show (Theorem \ref{couo2}) that the uo convergence is weaker than the compact open convergence iff the underlying topological space contains a dense locally compact subspace.\medskip

In Section \ref{ucpos} we start with the fact (Proposition \ref{loc0}) that the set of ``local'' order continuity of a continuous positive operator is a closed ideal  (with respect to the given locally solid convergence). From there we deduce (Proposition \ref{oc}) that order continuity of a positive operator can be propagated from an order dense sublattice which generates a dense ideal. This result is utilized in Section \ref{crss} in order to show that for a large class of convergences the closure of a regular sublattice is regular, with the original sublattice being order dense in the closure (Theorem \ref{rod} and Corollary \ref{reg2}). Furthermore, order continuity of a positive continuous operator can be propagated from a dense regular sublattice (Corollary \ref{doc}). We also present an example of a regular sublattice of a locally solid topological vector lattice whose closure is not regular (Example \ref{nonr}).\medskip

In Section \ref{sigs} we focus on the uniform convergence. We show that order continuity of a positive operator on a vector lattice can be propagated from a uniformly dense regular sublattice (Corollary \ref{uoc}; note that no additional ``external'' continuity of the operator is assumed). We also consider $\sigma$-order continuity (Theorem \ref{main1}): if a homomorphism is $\sigma$-order continuous on a sublattice, it is $\sigma$-order continuous on its closure, provided that either the domain or the co-domain has the $\sigma$-property (which every Banach lattice has). Note that we do not assume regularity of the sublattice in the last result. A discussion of some other facts available for $\sigma$-order convergence and continuity constitutes the final Section \ref{ssig}.\bigskip

\textbf{Convergence structures. }Recall that a \emph{net} in a set $X$ is a map from a pre-ordered (endowed with a reflexive transitive relation $\le$) directed (any two elements have an upper bound) set into $X$. \emph{A (net) convergence structure} is an ``adjudicator'' of whether a given net on $X$ converges to a given element of $X$, that satisfies the following axioms:

\begin{itemize}
\item Any constant net converges to its value.
\item If $\left(y_{\beta}\right)_{\beta\in B}$ is a quasi-subnet of $\left(x_{\alpha}\right)_{\alpha\in A}$ and $x_{\alpha}\to x$, then $y_{\beta}\to x$.
\item If $x_{\alpha}\to x$, $y_{\alpha}\to x$, and $\left(z_{\alpha}\right)_{\alpha\in A}$ is such that $z_{\alpha}\in\left\{x_{\alpha},y_{\alpha}\right\}$, for every $\alpha$, then $z_{\alpha}\to x$.
\end{itemize}

Recall that a net $\left(y_{\beta}\right)_{\beta\in B}$ is a \emph{quasi-subnet} of a net $\left(x_{\alpha}\right)_{\alpha\in A}$, if for every $\alpha_{0}\in A$ there is $\beta_{0}\in B$ such that $\left\{y_{\beta}\right\}_{\beta\ge \beta_{0}}\subset \left\{x_{\alpha}\right\}_{\alpha\ge \alpha_{0}}$, i.e. every tail of $\left(x_{\alpha}\right)_{\alpha\in A}$ contains a tail of $\left(y_{\beta}\right)_{\beta\in B}$. Note that every net is a quasi-subnet of any of its tails. Every subnet is a quasi-subnet. All constant nets with the same value are quasi-subnets of one another. Some set-theoretical subtleties associated with the definition of a convergence structure are discussed in detail in \cite{dow}.

We will slightly deviate from their terminology, by distinguishing the closure and adherence. The \emph{adherence} $\overline{A}^{1}$ of $A\subset X$ is the set of limits of all convergent nets in $A$. If $\overline{A}^{1}=A$, we say that $A$ is \emph{closed}, and the \emph{closure} $\overline{A}$ of $A$ is the intersection of all closed sets that contain $A$ (which is itself closed). Clearly, $A\subset\overline{A}^{1}\subset\overline{A}$. Axioms of a net convergence space imply that $\overline{A\cup B}^{1}=\overline{A}^{1}\cup \overline{B}^{1}$ and $\overline{A\cup B}=\overline{A}\cup \overline{B}$, for any $B\subset X$. It follows that the collection of all closed sets is a (closed) topology. We will say that $A$ is (topologically) dense if $\overline{A}^{1}=X$ ($\overline{A}=X$).

Since the adherence is not idempotent, we will often consider the iterated adherence defined for an ordinal $n$ as $\overline{A}^{n}:=\overline{\overline{A}^{n-1}}^{1}$, if $n$ has a predecessor, and $\overline{A}^{n}:=\bigcup\limits_{m<n}\overline{A}^{m}$, if $n$ is a limit ordinal. Using cardinality argument it is not hard to show that $\overline{A}^{n}=\overline{A}$, for a large enough $n$.\medskip

Let us conclude this section with some more natural definitions. First, we call a convergence space $X$ \emph{Hausdorff} if every net in $X$ has at most one limit. A map $\varphi:X\to Y$ into a convergence space $Y$ is \emph{continuous} if  $\varphi\left(x_{\alpha}\right)\to \varphi\left(x\right)$, whenever $x_{\alpha}\to x$. It is easy to see that $\varphi\left(\overline{A}^{1}\right)\subset \overline{\varphi\left(A\right)}^{1}$, for every $A\subset X$, and that a pre-image of a closed set with respect to a continuous map is closed, and so $\varphi$ is continuous with respect to the topologies generated by the convergences on $X$ and $Y$. We will also call $\varphi$ an \emph{embedding} if it is an injection and  $\varphi\left(x_{\alpha}\right)\to \varphi\left(x\right)$ $\Leftrightarrow$ $x_{\alpha}\to x$.

We will say that a convergence structure $\eta$ on $X$ is \emph{stronger} than the convergence structure $\theta$ if the identity map is continuous from $\left(X,\eta\right)$ into $\left(X,\theta\right)$. In this case $\overline{A}^{1}_{\eta}\subset\overline{A}^{1}_{\theta}$, for every $A\subset X$. A product convergence on $X\times Y$ is defined by declaring that $\left(x_{\alpha},y_{\alpha}\right)\to \left(x,y\right)$ if $x_{\alpha}\to x$ and $y_{\alpha}\to y$. Note that if $x_{\alpha}\to x$ and $y_{\beta}\to y$, then the double indexed net $\left(x_{\alpha},y_{\beta}\right)$ also converges to $\left(x,y\right)$ (see \cite[Remark 2.5]{dow}). Lastly, if a convergence is induced by a topology, we will not draw a distinction between the topology and the convergence.

\section{Locally solid convergence}\label{lscs}

In this section we consider the convergence structures on vector lattices which are compatible with the order and algebraic operations. Let $F$ be a vector lattice, whose positive cone is denoted by $F_{+}$. For any set $G\subset F$ we will denote $G_{+}=G\cap F_{+}$.

A linear subspace $E$ of $F$ is called a \emph{sublattice} if it is stable with respect to the lattice operations (or equivalently contains $\left|e\right|$, for each $e\in E$). Also, recall that $G\subset F$ is called \emph{solid} if for every $f\in F$ and $g\in G$ such that $\left|f\right|\le\left|g\right|$ it follows that $f\in G$. An \emph{ideal} is a solid subspace, or equivalently a sublattice for which $0\le f\le e\in E$ implies $f\in E$. It is not hard to show that a solid $E
\subset F$ is an ideal iff $E_{+}+E_{+}\subset E$.

Note that the intersection of sublattices / solids / ideals is a sublattice / solid / ideal. Accordingly, we will denote the ideal generated by $G\subset F$ by $I\left(G\right)$ (we will not use sublattices or solids generated by a set). A \emph{principal ideal} is an ideal of the form $F_{e}=I\left(\left\{e\right\}\right)=\bigcup\limits_{r\ge 0}\left[-re,re\right]$ -- the minimal ideal that contains $e\in F_{+}$.\medskip

A convergence structure on a vector lattice $F$ is called \emph{locally solid linear}, if the addition and scalar multiplication are continuous maps from $F\times F$ and $F\times \R$, respectively, into $F$, and whenever $\left(e_{\alpha}\right)_{\alpha\in A}$ and $\left(f_{\alpha}\right)_{\alpha\in A}$ are two nets in $F$ such that $\left|e_{\alpha}\right|\le \left|f_{\alpha}\right|$, for every $\alpha$, and $f_{\alpha}\to 0_{F}$, then $e_{\alpha}\to 0_{F}$. It is not hard to show that the lattice operations are continuous on $F$. Some general properties of locally solid convergences were investigated in \cite{aeg} and \cite{vw}. In particular, it was proven in \cite[Proposition 2.7]{vw} that the adherence of a sublattice / solid / ideal is a set of the same type. Using induction, this result can be extended to the closure. Also, the author has shown in \cite[Proposition 2.8]{vw} that if $F$ is Hausdorff, then it is \emph{Archimedean} (i.e. such that $\bigwedge\frac{1}{n}f=0_{F}$, for every $f\in F_{+}$) and order intervals are closed. One can also show that if in this case $G\subset F$ is a solid or a sublattice, then $\overline{G_{+}}^{1}=\overline{G}^{1}_{+}$ and $\overline{G_{+}}=\overline{G}_{+}$.\medskip

If the continuity of the scalar multiplication is replaced with the continuity of the additive inversion $f\mapsto -f$, we will call such a convergence locally solid \emph{additive}. For such convergences the previous comment about the adherences and the closures of solids and ideals (but not sublattices) is still valid. It turns out that relatively limited amount of data completely determines a locally solid additive convergence.

\begin{theorem}\label{locs}
Assume that $\to 0_{F}$ is a non-trivial adjudicator on $F_{+}$ to $0_{F}$ which satisfies the following properties:
\begin{itemize}
\item A quasi-subnet of a net convergent to $0_{F}$ converges to $0_{F}$;
\item If $\left(f_{\alpha}\right)_{\alpha\in A}\subset F_{+}$ is convergent to $0_{F}$, then $g_{\alpha}\to 0_{F}$, where $0_{F}\le g_{\alpha}\le f_{\alpha}$, for every $\alpha\in A$;
\item If $f_{\alpha}\to 0_{F}$ and $g_{\alpha}\to 0_{F}$, then $f_{\alpha}+g_{\alpha}\to 0_{F}$.
\end{itemize}
Then, the convergence defined by $f_{\alpha}\to f$ if $\left|f_{\alpha}-f\right|\to 0_{F}$ is a locally solid additive convergence. It is linear iff $\frac{1}{n}f\to 0_{F}$, for every $f\in F_{+}$.
\end{theorem}
\begin{proof}
First, since $\to 0_{F}$ is non-trivial, there is a net $\left(f_{\alpha}\right)_{\alpha\in A}\subset F_{+}$ which converges to $0_{F}$. Then, the net $\left(g_{\alpha}\right)_{\alpha\in A}$ defined by $g_{\alpha}=0_{F}$, for every $\alpha$, satisfies $0_{F}\le g_{\alpha}\le f_{\alpha}$, hence $g_{\alpha}\to 0_{F}$. Since every zero-net is a quasi-subnet of $\left(g_{\alpha}\right)_{\alpha\in A}$, it follows that every constant net converges to its value.\medskip

Assume that $f_{\alpha}\to f$ and $g_{\alpha}\to g$, where $\left(f_{\alpha}\right)_{\alpha\in A},\left(g_{\alpha}\right)_{\alpha\in A}\subset F$. Then\linebreak $\left|\left(f_{\alpha}+g_{\alpha}\right)-\left(f+g\right)\right|\le \left|f_{\alpha}-f\right|+\left|g_{\alpha}-g\right|\to 0_{F}$, and so $f_{\alpha}+g_{\alpha}\to f+g$. Hence, the addition is continuous. Continuity of the additive inversion is easy to see.\medskip

Assume that $f_{\alpha}\to f\leftarrow g_{\alpha}$ and $\left(e_{\alpha}\right)_{\alpha\in A}\subset F$ is such that $e_{\alpha}\in\left\{f_{\alpha},g_{\alpha}\right\}$, for every $\alpha\in A$. Since the continuity of the operations was already established, $f_{\alpha}-g_{\alpha}\to 0_{F}$, and for every $\alpha\in A$ we have $f_{\alpha}-e_{\alpha}\in\left\{f_{\alpha}-g_{\alpha},0_{F}\right\}$. Hence, $\left|f_{\alpha}-e_{\alpha}\right|\le \left|f_{\alpha}-g_{\alpha}\right|\to 0_{F}$, from where $f_{\alpha}-e_{\alpha}\to 0_{F}$, and so $e_{\alpha}=f_{\alpha}-\left(f_{\alpha}-e_{\alpha}\right)\to f-0_{F}=f$. Thus, $\to$ is a convergence structure on $F$, which is clearly locally solid additive.\medskip

Now let us show that if $\frac{1}{n}f\to 0_{F}$, for every $f\in F_{+}$, the convergence is linear. First, it is easy to show using the given axioms that the sequence $f,0_{F},\frac{1}{2}f,0_{F},...$ converges to $0_{F}$.

Let $\left(f_{\alpha}\right)_{\alpha\in A}\subset F$ and $\left(\lambda_{\alpha}\right)_{\alpha\in A}\subset \R$ be such that $f_{\alpha}\to f$ and $\lambda_{\alpha}\to \lambda$. Since $\lambda_{\alpha}\to \lambda$, by passing to a tail we may assume that $\left|\lambda_{\alpha}-\lambda\right|\le 1$, for every $\alpha$. Define $\left(\mu_{\alpha}\right)_{\alpha\in A}\subset \left[0,1\right]$ by $\mu_{\alpha}^{-1}=\lfloor\left|\lambda_{\alpha}-\lambda\right|^{-1}\rfloor$, i.e. $\frac{1}{\mu_{\alpha}}$ is the integer part of $\left|\lambda_{\alpha}-\lambda\right|^{-1}$, for every $\alpha$ (in the event that $\lambda_{\alpha}=\lambda$ define $\mu_{\alpha}=0$). It then follows that $\left(\mu_{\alpha}\left|f\right|\right)_{\alpha\in A}$ is a quasi-subnet of $\left|f\right|,0_{F},\frac{1}{2}\left|f\right|,0_{F},...$, and so $\mu_{\alpha}\left|f\right|\to 0_{F}$. On the other hand $\left|\left(\lambda_{\alpha}-\lambda\right)f\right|\le \mu_{\alpha}\left|f\right|$, from where $\left(\lambda_{\alpha}-\lambda\right)f\to 0_{F}$.

We may assume that $\left(\lambda_{\alpha}\right)_{\alpha\in A}\subset \left[-m,m\right]$, for some $m\in\N$. From local solidness it follows that $\left|\frac{1}{m}\lambda_{\alpha}\left(f_{\alpha}-f\right)\right|\le \left|f_{\alpha}-f\right|\to 0_{F}$, hence $\frac{1}{m}\lambda_{\alpha}\left(f_{\alpha}-f\right)\to 0_{F}$. Applying additivity $m-1$ times yields $\lambda_{\alpha}\left(f_{\alpha}-f\right)\to 0_{F}$. It now follows that $\lambda_{\alpha}f_{\alpha}-\lambda f=\lambda_{\alpha}\left(f_{\alpha}-f\right)+\left(\lambda_{\alpha}-\lambda\right)f\to 0_{F}$.
\end{proof}

Let us show that the continuity of an \emph{operator} (=linear map) is determined by its continuity on $F_{+}$ at $0_{F}$. In particular, for locally solid additive convergences $\eta,\theta$ on $F$ the former is stronger than the latter iff $0_{F}\le f_{\alpha}\xrightarrow[]{\eta} 0_{F}$ implies $f_{\alpha}\xrightarrow[]{\theta} 0_{F}$.

\begin{proposition}\label{contop}
If $T:F\to E$ is an operator between vector lattices which are endowed with locally solid additive convergences, then $T$ is continuous iff it is continuous on $F_{+}$ at $0_{F}$.
\end{proposition}
\begin{proof}
If $f_{\alpha}\to 0_{F}$, then $0_{F}\le f_{\alpha}^{\pm}\to 0_{F}$, from where $Tf_{\alpha}^{\pm}\to 0_{E}$, and so $Tf_{\alpha}=Tf_{\alpha}^{+}-Tf_{\alpha}^{-}\to 0_{E}$.
\end{proof}

Consider an example of a locally solid linear convergence.

\begin{example}\label{ru}
For $e\in F_{+}$ the principal ideal $F_{e}$ is endowed with the norm $\|\cdot\|_{e}$ defined by $\|f\|_{e}=\bigwedge\left\{\alpha\ge 0:~ \left|f\right|\le\alpha e\right\}$. A net $\left(f_{\alpha}\right)_{\alpha\in A}\subset F$ converges \emph{uniformly} to $f$ \emph{relative} to $e\in F$ if $f\in F_{e}$ and for every $\varepsilon>0$ there is $\alpha_{0}$ such that $\|f-f_{\alpha}\|_{e}\le\varepsilon$ (and in particular $f_{\alpha}\in F_{e}$), for every $\alpha\ge \alpha_{0}$. We denote it $f_{\alpha}\xrightarrow[]{\|\cdot\|_{e}}f$ and say that $\left(f_{\alpha}\right)_{\alpha\in A}\subset F$ converges \emph{uniformly} to $f$, if $f_{\alpha}\xrightarrow[]{\|\cdot\|_{e}}f$, for some $e$. It is easy to see that the uniform convergence is locally solid linear. Moreover, it is Hausdorff iff $F$ is Archimedean. Recall that $F$ has $\sigma$\emph{-property} if every countable set is contained in a principal ideal. This property is equivalent to the fact that the adherence of every set is closed, and so is equal to the closure (see \cite[Theorem 72.3]{zl}).
\qed\end{example}

We will say that a locally solid additive convergence $\eta$ on $F$ is \emph{idempotent} if for any nets $\left(f_{\alpha}\right)_{\alpha\in A},\left(g_{\beta}\right)_{\beta\in B}\subset F_{+}$ such that $g_{\beta}\xrightarrow[\beta]{} 0_{F}$ and $\left(f_{\alpha}-g_{\beta}\right)^{+}\xrightarrow[\alpha]{} 0_{F}$, for every $\beta$, it follows that $f_{\alpha}\xrightarrow[\alpha]{} 0_{F}$. Note that the corresponding notion of the product of locally solid additive convergences will be explored in the future works. Idempotent convergences come handy in the following auxiliary result.

\begin{proposition}\label{umb}
Let $\eta$ be a locally solid additive convergence on $F$. If $\left(f_{\alpha}\right)_{\alpha\in A}$ is a net in $F_{+}$, then $H=\left\{h\in F:~ \left|h\right|\wedge f_{\alpha}\to 0_{F}\right\}$ is an ideal. If $\eta$ is idempotent, then $H$ is closed.
\end{proposition}
\begin{proof}
If $h\in H$ and $\left|g\right|\le \left|h\right|$, we have $\left|g\right|\wedge f_{\alpha}\le \left|h\right|\wedge f_{\alpha}\to 0_{F}$, from where $\left|g\right|\wedge f_{\alpha}\to 0_{F}$, and so $g\in H$. If $g,h\in H_{+}$, then $\left(g+h\right)\wedge f_{\alpha}\le g\wedge f_{\alpha}+h\wedge f_{\alpha}\to 0_{F}$, hence $g+h\in H$. Therefore, $H$ is an ideal.\medskip

Now let $h\in \overline{H}^{1}_{+}$, so that there is a net $\left(h_{\beta}\right)_{\beta\in B}\subset H$ which converges to $h$. Replacing $h_{\beta}$ with $h_{\beta}^{+}\wedge h$ if needed we may assume that $\left(h_{\beta}\right)_{\beta\in B}\subset\left[0_{H},h\right]$. Then, $h-h_{\beta}\to 0_{F}$, and for every $\beta$ we have $\left(h\wedge f_{\alpha}-h+h_{\beta}\right)^{+}\le h_{\beta}\wedge f_{\alpha}\xrightarrow[\alpha]{} 0_{F}$. If the convergence is idempotent, this implies that $h\wedge f_{\alpha}\to 0_{F}$, and so $h\in H$. Thus, in this case $H$ is closed.
\end{proof}

Let us now show that idempotent convergences form a wide class (see also Proposition \ref{ocls5}).

\begin{example}\label{topid}
Every locally solid additive topology is idempotent. Assume that $g_{\beta}\xrightarrow[\beta]{} 0_{F}$ and $\left(f_{\alpha}-g_{\beta}\right)^{+}\xrightarrow[\alpha]{} 0_{F}$, for every $\beta$. Take a solid neighborhood $U$ of $0_{F}$, and a neighborhood $V$ of $0_{F}$ such that $V+V\subset U$. There is $\beta\in B$ such that $g_{\beta}\in V$, and there is $\alpha_{0}$ such that $\left(f_{\alpha}-g_{\beta}\right)^{+}\in V$, for every $\alpha\ge\alpha_{0}$. Hence, $f_{\alpha}\le \left(f_{\alpha}-g_{\beta}\right)^{+}+g_{\beta}\in V+V\subset U$, and so $f_{\alpha}\in U$, for every $\alpha\ge\alpha_{0}$. Thus, $f_{\alpha}\to 0_{F}$.
\qed\end{example}

\begin{example}
If $F$ has $\sigma$-property the uniform convergence is idempotent. Assume that $g_{\beta}\xrightarrow[\beta]{} 0_{F}$ and $\left(f_{\alpha}-g_{\beta}\right)^{+}\xrightarrow[\alpha]{} 0_{F}$, for every $\beta$. There is $e_{0}$ such that $\|g_{\beta}\|_{e_{0}}\to 0$. Choose a sequence $\left(\beta_{n}\right)_{n\in\N}\subset B$ such that $\|g_{\beta_{n}}\|_{e_{0}}\to 0$. For every $n\in\N$ there is $e_{n}$ such that $\|\left(f_{\alpha}-g_{\beta_{n}}\right)^{+}\|_{e_{n}}\xrightarrow[\alpha]{}  0$. From $\sigma$-property, there is $e\in F_{+}$ such that $\left\{e_{n}\right\}_{n\in\N\cup\left\{0\right\}}\subset F_{e}$; it follows that $\|\left(f_{\alpha}-g_{\beta_{n}}\right)^{+}\|_{e}\xrightarrow[\alpha]{} 0$ and $\|g_{\beta_{n}}\|_{e}\xrightarrow[\alpha]{n\in\N}  0$. Arguing as in Example \ref{topid} it is easy to show that $\|f_{\alpha}\|_{e}\to 0$, from where $f_{\alpha}\to 0_{F}$.
\qed\end{example}

\section{Unbounded modification}\label{ums}

In this section $F$ is a vector lattice endowed with a locally solid additive convergence $\eta$. We will consider the unbounded modification of $\eta$, generalizing some of the results from e.g. \cite{acw2}, \cite{aeg}, \cite{dem1}, \cite{pap1} and \cite{taylor}. For $A\subset F$ define a convergence $u_{A}\eta$ on $F$ by $0_{F}\le f_{\alpha}\xrightarrow[]{u_{A}\eta} 0_{F}$ if $\left|a\right|\wedge f_{\alpha}\xrightarrow[]{\eta} 0_{F}$, for every $a\in A$. We will also denote $u\eta=u_{F}\eta$. According to Theorem \ref{locs}, this data indeed determines a locally solid additive convergence on $F$, which is weaker than $\eta$. It is also clear that if $A\subset B$, then $u_{A}\eta$ is weaker than $u_{B}\eta$. In the next proposition we gather the basic properties of the unbounded modification.

\begin{proposition}\label{unba}Let $A\subset F$ and let $E,G$ be ideals in $F$. Then:
\item[(i)] $u_{A}\eta=u_{I\left(A\right)}\eta$ is the weakest locally solid additive convergence on $F$, which is stronger than (in fact coincides with) $\eta$ on $\left[0_{F},\left|a\right|\right]$, for every $a\in A$.
\item[(ii)] $0_{F}\le f_{\alpha}\xrightarrow[]{u_{A}\eta} 0_{F}$ iff $g_{\alpha}\xrightarrow[]{\eta} 0_{F}$, for every net $\left(g_{\alpha}\right)_{\alpha\in A}\subset \left[0_{F},\left|a\right|\right]$, $a\in A$, with $0_{F}\le g_{\alpha}\le f_{\alpha}$, for every $\alpha$.
\item[(iii)] $u_{G}u_{E}\eta=u_{E\cap G}\eta$. In particular, $u_{E}u\eta=uu_{E}\eta=u_{E}u_{E}\eta=u_{E}\eta$.
\item[(iv)] $f_{\alpha}\xrightarrow[]{u_{E}\eta} f$, iff $g\vee f_{\alpha}\wedge h\xrightarrow[]{\eta} g\vee f\wedge h$, for every $g,h$ with $h-g\in E_{+}$. In particular, $0_{F}\le f_{\alpha}\xrightarrow[]{u\eta} f\in F_{+}$ iff $h\wedge f_{\alpha}\xrightarrow[]{\eta} h\wedge f$, for every $h\in F_{+}$.
\item[(v)] If $\eta$ is idempotent, then $u_{A}\eta=u_{\overline{I\left(A\right)}}\eta$.
\item[(vi)] If $\eta$ is idempotent and $h\in F_{+}$ is a topological unit (i.e. $F_{h}$ is topologically dense in $F$), then for a net $\left(f_{\alpha}\right)_{\alpha\in A}\subset F_{+}$ we have $f_{\alpha}\xrightarrow[]{u\eta} 0_{F}$ if and only if $h \wedge f_{\alpha}\xrightarrow[]{\eta} 0_{F}$.
\end{proposition}
\begin{proof}
(i): It is clear that $u_{I\left(A\right)}\eta$ is stronger than $u_{A}\eta$. To prove the converse, assume that $0_{F}\le f_{\alpha}\xrightarrow[]{u_{A}\eta} 0_{F}$ and let $H=\left\{h\in F:~ \left|h\right|\wedge f_{\alpha}\to 0_{F}\right\}$. By assumption, $A\subset H$, and according to Proposition \ref{umb}, $H$ is an ideal, therefore it contains $I\left(A\right)$. Hence, $\left|h\right|\wedge f_{\alpha}\to 0_{F}$, for every $h\in I\left(A\right)$, and so $f_{\alpha}\xrightarrow[]{u_{I\left(A\right)}\eta} 0_{F}$.\medskip

Let $a\in A$. If $\left[0_{F},\left|a\right|\right]\ni f_{\alpha}\xrightarrow[]{u_{A}\eta} f\in \left[0_{F},\left|a\right|\right]$, then $\frac{1}{2}f_{\alpha}\xrightarrow[]{u_{A}\eta} \frac{1}{2}f$, therefore $\frac{1}{2}\left|f_{\alpha}-f\right|=\left|a\right|\wedge\left|\frac{1}{2}f_{\alpha}-\frac{1}{2}f\right|\xrightarrow[]{\eta}0_{F}$, from where $f_{\alpha}\xrightarrow[]{\eta}f$. Hence, $u_{A}\eta$ and $\eta$ coincide on $\left[0_{F},\left|a\right|\right]$.\medskip

Assume that $\theta$ is a locally solid additive convergence on $F$, which is stronger than $\eta$ on $\left[0_{F},\left|a\right|\right]$, for every $a\in A$. If $0_{F}\le f_{\alpha}\xrightarrow[]{\theta}0_{F}$, then for every $a\in A$ we have that $\left[0_{F},\left|a\right|\right]\ni \left|a\right|\wedge f_{\alpha}\xrightarrow[]{\theta}0_{F}$, hence $\left|a\right|\wedge f_{\alpha}\xrightarrow[]{\eta}0_{F}$. Since $a$ was arbitrary we conclude that $f_{\alpha}\xrightarrow[]{u_{A}\eta}0_{F}$, and so $\theta$ is stronger than $u_{A}\eta$.\medskip

(iii): Let $\left(f_{\alpha}\right)_{\alpha\in A}\subset F_{+}$. Then $f_{\alpha}\xrightarrow[]{u_{G}u_{E}\eta} 0_{F}$ iff $g\wedge f_{\alpha}\xrightarrow[]{u_{E}\eta} 0_{F}$, for every $g\in G_{+}$ iff $e\wedge g\wedge f_{\alpha}\xrightarrow[]{\eta} 0_{F}$, for every $g\in G_{+}$ and $e\in E_{+}$. Since $\left\{e\wedge g:~ e\in E_{+},~ g\in G_{+}\right\}=E\cap G_{+}$, the last condition is equivalent to $f\wedge f_{\alpha}\xrightarrow[]{\eta} 0_{F}$, for every $f\in E\cap G_{+}$, which in turn means that $f_{\alpha}\xrightarrow[]{u_{E\cap G}\eta} 0_{F}$.\medskip

(iv): Let us prove the first claim. Sufficiency: For every $e\in E_{+}$ it follows that $$e\wedge\left(f_{\alpha}-f\right)^{+}=\left(e+f\right)\wedge f_{\alpha}\vee f-f\xrightarrow[]{\eta}f-f=0_{F},$$ and so $\left(f_{\alpha}-f\right)^{+}\xrightarrow[]{u_{E}\eta}0_{F}$, and analogously $\left(f_{\alpha}-f\right)^{-}\xrightarrow[]{u_{E}\eta}0_{F}$, from where $f_{\alpha}\xrightarrow[]{u_{E}\eta} f$.\medskip

Necessity: Fix $g,h$ with $h-g\in E_{+}$. From continuity of the lattice operations in the locally solid additive convergence $u_{E}\eta$, it follows that $f_{\alpha}\xrightarrow[]{u_{E}\eta} f$ implies $\left(f_{\alpha}-g\right)^{+}\wedge \left(h-g\right)\xrightarrow[]{u_{E}\eta}\left(f-g\right)^{+}\wedge \left(h-g\right)$. Since according to (i) $\eta$ and $u_{E}$ coincide on $\left[0_{F}, h-g\right]$  , it follows that $$g\vee f_{\alpha}\wedge h-g=\left(f_{\alpha}-g\right)^{+}\wedge \left(h-g\right)\xrightarrow[]{\eta}\left(f-g\right)^{+}\wedge \left(h-g\right)=g\vee f\wedge h-g,$$ from where $g\vee f_{\alpha}\wedge h\xrightarrow[]{\eta} g\vee f\wedge h$.\medskip

Now let us deal with the second claim. Necessity is derived from the first claim by taking $g=0_{F}$. For sufficiency, take $g\le h$ and observe that $h\wedge f_{\alpha}=h^{+}\wedge f_{\alpha}-h^{-}\xrightarrow[]{\eta}h^{+}\wedge f-h^{-}=h\wedge f$, from where and continuity of the operations we get
$h\wedge f_{\alpha}\vee g\xrightarrow[]{\eta}h\wedge f\vee g$. Hence, by the first claim $f_{\alpha}\xrightarrow[]{u\eta} f$.\medskip

(ii) is easy to inspect. (v) is proven similarly to the first claim in (i), using the fact that according to Proposition \ref{umb}, $H$ is a closed ideal. (vi) immediately follows from (v).
\end{proof}

In particular, if $\eta$ is Hausdorff, for monotone nets $\eta$ and $u\eta$ convergences coincide, since monotone $u\eta$-convergent nets are eventually order bounded. We will call a locally solid convergence $\eta$ \emph{unbounded} if $\eta=u\eta$, or equivalently if $\eta=u\theta$, for some locally solid convergence $\theta$. Part (iv) of Proposition \ref{unba} shows that unbounded convergences are initial with respect to a certain collection of maps on $F$.

\begin{example}
For a Tychonoff space $X$ the compact-open topology $\tau$ (i.e. the uniform convergence on the compact sets) on $\Co\left(X\right)$ is unbounded (it is clear that it is locally solid linear). Indeed, since $\1$ is a topological unit, it follows that $\0\le f_{\alpha}\xrightarrow[]{u\tau}\0$ iff $\1\wedge f_{\alpha}\xrightarrow[]{\tau}\0$, which is easily seen to be equivalent to $f_{\alpha}\xrightarrow[]{\tau}\0$.
\qed\end{example}

The following result is a refinement of part (iv) of Proposition \ref{unba}.

\begin{proposition}\label{unbaa}If $\eta$ is linear and idempotent, then $f_{\alpha}\xrightarrow[]{u_{E}\eta} f$, iff $e\vee f_{\alpha}\wedge h\xrightarrow[]{\eta} e\vee f\wedge h$, for every $e,h\in E$ with $e\le h$. In particular, $0_{F}\le f_{\alpha}\xrightarrow[]{u_{E}\eta} f\in F_{+}$ iff $h\wedge f_{\alpha}\xrightarrow[]{\eta} h\wedge f$, for every $h\in E_{+}$.
\end{proposition}
\begin{proof}
Necessity in both claims follows from (iv) of Proposition \ref{unba}. Let us prove the sufficiency in the second claim first.

Fix $e\in E_{+}$ and note that $e\wedge f_{\alpha}\xrightarrow[]{\eta} e\wedge f$ implies $\left(e-e\wedge f\right)\wedge \left(f_{\alpha}-e\wedge f\right)\xrightarrow[]{\eta}0_{F}$. Taking the positive part of the left hand side we get that $\left(e-f\right)^{+}\wedge \left(f_{\alpha}-f\right)^{+}\le\left(e-f\right)^{+}\wedge \left(f_{\alpha}-e\wedge f\right)^{+}\xrightarrow[]{\eta}0_{F}$, hence $\left(e-f\right)^{+}\wedge \left(f_{\alpha}-f\right)^{+}\xrightarrow[]{\eta}0_{F}$. Taking the negative part yields
$$\left(e-f\right)^{+}\wedge \left(f-f_{\alpha}\right)^{+}\le \left(e-f_{\alpha}\wedge f\right)^{+}\wedge \left(f-f_{\alpha}\wedge e\right)^{+}=e\wedge f-f_{\alpha}\wedge e\wedge f=\left(f_{\alpha}-e\wedge f\right)^{-}\xrightarrow[]{\eta}0_{F},$$ and so $\left(e-f\right)^{+}\wedge \left(f_{\alpha}-f\right)^{-}\xrightarrow[]{\eta}0_{F}$. It now follows that $H=\left\{h\in F:~ \left|h\right|\wedge \left|f_{\alpha}-f\right|\to 0_{F}\right\}$ is a closed ideal, which includes $\left(e-f\right)^{+}$, for every $e\in E_{+}$. Then, $e\wedge\left(ne-f\right)^{+}\in H$, for every $n\in\N$. Moreover,
\begin{align*}
e-e\wedge\left(ne-f\right)^{+}&=\left(e-\left(ne-f\right)^{+}\right)^{+}=\left(e\wedge \left(f-\left(n-1\right)e\right)\right)^{+}=e\wedge \left(f-\left(n-1\right)e\right)^{+}\\
&\le \left(e-\frac{1}{n-1}f\right)^{+}\wedge \left(f-\left(n-1\right)e\right)^{+}+\frac{1}{n-1}f=\frac{1}{n-1}f\xrightarrow[]{\eta}0_{F},
\end{align*}
from where $e\wedge\left(ne-f\right)^{+} \xrightarrow[]{\eta}e$, and so $e\in H$. Thus, $E\subset H$, and so $f_{\alpha}\xrightarrow[]{u_{E}\eta}f$.\medskip

To prove the sufficiency in the first claim, taking $e=0_{F}$ yields that $f_{\alpha}^{+}\wedge h\xrightarrow[]{\eta} f^{+}\wedge h$, for every $h\in E_{+}$. Hence, in the light of the second claim, $f_{\alpha}^{+}\xrightarrow[]{u_{E}\eta} f^{+}$. Similarly, $f_{\alpha}^{-}\xrightarrow[]{u_{E}\eta} f^{-}$, and so $f_{\alpha}\xrightarrow[]{u_{E}\eta} f$.
\end{proof}

\begin{remark}In Proposition \ref{unbaa} the requirement of linearity of $\eta$ cannot be dropped even if it is topological (see \cite[Example 4.10]{acw2}). We do not have a linear but non-idempotent counterexample.
\qed\end{remark}

\section{Order convergence in a vector lattice}\label{ocvls}

In this section $F$ is an Archimedean vector lattice. A net $\left(f_{\alpha}\right)_{\alpha\in A}\subset F$ is \emph{increasing} if $\alpha\le \beta$ implies $f_{\alpha}\le f_{\beta}$. In this case the set $\left\{f_{\alpha}\right\}_{\alpha\in A}$ is directed, and conversely, any directed subset of $F$ may be viewed as an increasing net indexed by itself. Decreasing nets are defined analogously. For such a net $\left\{f_{\alpha}\right\}_{\alpha\in A}$ is directed downward. A net $\left(f_{\alpha}\right)_{\alpha\in A}\subset F_{+}$ converges \emph{in order} to $0_{F}$ if there is $G\subset F$ with $\bigwedge G=0_{F}$ which \emph{dominates the tails} of $\left(f_{\alpha}\right)_{\alpha\in A}$, i.e. for each $g\in G$ there is $\alpha_{0}$ such that $f_{\alpha}\le g$, for every $\alpha\ge \alpha_{0}$. Note that $G$ can always be chosen directed downward (and so considered a decreasing net), by replacing it with the set $G^{\wedge}=\left\{g_{1}\wedge ...\wedge g_{n},~ g_{1},...,g_{n}\in G\right\}$ (note that $G^{\vee}$ is defined analogously).

It is easy to see that if $H\subset F$ satisfies $\bigwedge H=0_{F}$, then $\bigwedge \left(G+H\right)=0_{F}$. Using this observation, one can show that Theorem \ref{locs} applies, and so we get a locally solid linear convergence on $F$. We denote this convergence by $f_{\alpha}\xrightarrow[]{o}f$. It is not hard to prove that every monotone net order converges to its extremum (if it exists). The following proposition explains the special role of the order convergence in the theory of vector lattices.

\begin{theorem}\label{ordex}The order convergence is the strongest locally solid additive convergence in which every monotone net converges to its supremum / infimum. That is, if $\eta$ is a locally solid additive convergence on $F$ such that $f_{\alpha}\downarrow 0_{F}$ implies $f_{\alpha}\xrightarrow[]{\eta}0_{F}$, then $f_{\alpha}\xrightarrow[]{o}f$ implies $f_{\alpha}\xrightarrow[]{\eta}f$.
\end{theorem}
\begin{proof}
We will prove that if a net $\left(f_{\alpha}\right)_{\alpha\in A}\subset F_{+}$ is order null, then $f_{\alpha}\xrightarrow[]{\eta}f$. Let $\left(g_{\beta}\right)_{\beta\in B}\subset F_{+}$ decrease to $0_{F}$ and dominate the tails of $\left(f_{\alpha}\right)_{\alpha\in A}$. Let us show that $\Gamma=\left\{\left(\alpha,\beta\right)\in A\times B,~ f_{\alpha}\le g_{\beta}\right\}$ is a directed set. Let $\left(\alpha_{1},\beta_{1}\right), \left(\alpha_{2},\beta_{2}\right)\in \Gamma$; since $B$ is directed, there is $\beta\ge\beta_{1},\beta_{2}$, and by our assumption there is $\alpha_{0}$ such that $0_{F}\le f_{\alpha}\le g_{\beta}$, for every $\alpha\ge\alpha_{0}$; since $A$ is directed, there is $\alpha\ge\alpha_{0},\alpha_{1},\alpha_{2}$; it follows that $f_{\alpha}\le g_{\beta}$, and so $\left(\alpha,\beta\right)\in \Gamma$ with $\left(\alpha,\beta\right)\ge \left(\alpha_{1},\beta_{1}\right), \left(\alpha_{2},\beta_{2}\right)$. For every $\gamma=\left(\alpha,\beta\right)\in \Gamma$ define $e_{\gamma}=f_{\alpha}$ and $h_{\gamma}=g_{\beta}$; by the definition of $\Gamma$ we have that $0_{F}\le e_{\gamma}\le h_{\gamma}$; since $h_{\gamma}\downarrow 0_{F}$, we have $e_{\gamma}\xrightarrow[]{\eta}0_{F}$. It is left to notice that $\left(f_{\alpha}\right)_{\alpha\in A}$ is a quasi-subnet of $\left(e_{\gamma}\right)_{\gamma\in \Gamma}$. Indeed, if $\gamma_{0}=\left(\alpha_{0},\beta_{0}\right)\in \Gamma$, there is $\alpha_{1}\ge \alpha_{0}$ such that $f_{\alpha}\le g_{\beta_{0}}$, for every $\alpha\ge\alpha_{1}$; hence, $\left\{f_{\alpha}\right\}_{\alpha\ge \alpha_{1}}\subset\left\{e_{\left(\alpha,\beta_{0}\right)}\right\}_{\alpha\ge \alpha_{1}}\subset\left\{e_{\gamma}\right\}_{\gamma\ge \left(\alpha_{0},\beta_{0}\right)}$. Thus, $f_{\alpha}\xrightarrow[]{\eta}0_{F}$.
\end{proof}

Recall that an operator $T:F\to E$ between vector lattices is called \emph{positive} if $TF_{+}\subset E_{+}$ and \emph{order continuous} if it is continuous with respect to the order convergences on $F$ and $E$. The following fact is well-known but, we would like to place it into the context of the specialness of the order convergence.

\begin{corollary}\label{idcont}
A positive operator $T:F\to E$ is order continuous if and only if it maps nets decreasing to $0_{F}$ into nets decreasing to $0_{E}$.
\end{corollary}
\begin{proof}
We only need to prove sufficiency. Consider the convergence $\eta$ on $F$ defined by $f_{\alpha}\xrightarrow[]{\eta} 0_{F}$, for $\left(f_{\alpha}\right)_{\alpha\in A}\subset F_{+}$, if $Tf_{\alpha}\xrightarrow[]{o} 0_{E}$. It is not hard to show that $\eta$ satisfies the conditions in Theorem \ref{locs} and hence determines a locally solid linear convergence on $F$. By assumption, every net which decreases to $0_{F}$ converges to $0_{F}$ with respect to $\eta$. Hence, according to Theorem \ref{ordex}, the order convergence is stronger than $\eta$, and so $T$ is order continuous.
\end{proof}

Let us now confirm that the concept of order convergence on a vector lattice is consistent with the concept of order convergence on a general partially ordered set (see e.g. \cite{acw}). The following auxiliary fact is often useful.

\begin{lemma}\label{arch}If $\varnothing\ne G$ be order bounded from below and $f\in F_{+}$ is such that $G-f\subset G$, then $f=0_{F}$.
\end{lemma}
\begin{proof}
Assume that $G\ge h\in F$. Since $G$ is nonempty, there is $g\in G$. Then, $g-f\in G$ and arguing by induction, it follows that $g- nf\in G\ge h$, and so $g-h\ge nf$, for every $n\in\N$. Since $F$ is Archimedean and $f\ge 0_{F}$, it follows that $f= 0_{F}$.
\end{proof}

We will call $f\in F$ an \emph{eventual upper bound} for a net $\left(f_{\alpha}\right)_{\alpha\in A}\subset F$ if there is $\alpha_{0}$ such that $f_{\alpha}\le f$, for every $\alpha\ge\alpha_{0}$. We will denote the collection of all eventual upper bounds of this net by $\left(f_{\alpha}\right)_{\alpha\in A}^{\nearrow}$. \emph{Eventual lower bounds} are defined analogously and their collection is denoted by $\left(f_{\alpha}\right)_{\alpha\in A}^{\searrow}$. Note that $\left(f_{\alpha}\right)_{\alpha\in A}^{\searrow}\le\left(f_{\alpha}\right)_{\alpha\in A}^{\nearrow}$. The following was first proven in \cite{pap1}.

\begin{proposition}[Papangelou]\label{ocs}Let $\left(f_{\alpha}\right)_{\alpha\in A}\subset F$ be an eventually bounded net and let $f\in F$. Then:
\item[(i)] $f_{\alpha}\xrightarrow[]{o}f$ iff $\bigvee\left(f_{\alpha}\right)_{\alpha\in A}^{\searrow}=f=\bigwedge\left(f_{\alpha}\right)_{\alpha\in A}^{\nearrow}$.
\item[(ii)] $\left(f_{\alpha}\right)_{\alpha\in A}^{\nearrow}\ge f$ iff $\bigvee\limits_{\alpha\ge\alpha_{0}} f\wedge f_{\alpha}=f$, for every $\alpha_{0}$.
\end{proposition}
\begin{proof}
(i): Sufficiency: Let $G=\left(f_{\alpha}\right)_{\alpha\in A}^{\nearrow}-\left(f_{\alpha}\right)_{\alpha\in A}^{\searrow}$; clearly, $\bigwedge G=0_{F}$, and for every $g=h-e$, where $h\in \left(f_{\alpha}\right)_{\alpha\in A}^{\nearrow}$ and $e\in \left(f_{\alpha}\right)_{\alpha\in A}^{\searrow}$, there is $\alpha_{0}$ such that $e\le f_{\alpha}\le h$, for every $\alpha\ge\alpha_{0}$. Since also $e\le f\le h$, it follows that $\left|f_{\alpha}-f\right|\le h-e=g$, hence $G$ dominates the tails of $\left(\left|f_{\alpha}-f\right|\right)_{\alpha\in A}$, therefore $\left|f_{\alpha}-f\right|\xrightarrow[]{o}0_{F}$, and so $f_{\alpha}\xrightarrow[]{o}f$.\medskip

Necessity: Let $G\subset F$ satisfy $\bigwedge G=0_{F}$ and dominate the tails of $\left(\left|f_{\alpha}-f\right|\right)_{\alpha\in A}$. The latter means that for every $g\in G$ there is $\alpha_{0}$ such that $\left|f_{\alpha}-f\right|\le g$, from where $f-g\le f_{\alpha}\le f+g$, for every $\alpha\ge\alpha_{0}$. Hence, $f-G\subset \left(f_{\alpha}\right)_{\alpha\in A}^{\searrow}$, while $f+G\subset \left(f_{\alpha}\right)_{\alpha\in A}^{\nearrow}$. Therefore, any upper bound of $\left(f_{\alpha}\right)_{\alpha\in A}^{\searrow}$ has to be above $f$, and every lower bound of $\left(f_{\alpha}\right)_{\alpha\in A}^{\nearrow}$ has to be below $f$. Since we also have $\left(f_{\alpha}\right)_{\alpha\in A}^{\searrow}\le\left(f_{\alpha}\right)_{\alpha\in A}^{\nearrow}$, it follows that $f$ is an upper bound for $\left(f_{\alpha}\right)_{\alpha\in A}^{\searrow}$, and a lower bound for $\left(f_{\alpha}\right)_{\alpha\in A}^{\nearrow}$. Thus, $\bigvee\left(f_{\alpha}\right)_{\alpha\in A}^{\searrow}=f=\bigwedge\left(f_{\alpha}\right)_{\alpha\in A}^{\nearrow}$.\medskip

(ii): Sufficiency: Let $g\in \left(f_{\alpha}\right)_{\alpha\in A}^{\nearrow}$. There is $\alpha_{0}$ such that $g\ge f_{\alpha}$, for every $\alpha\ge\alpha_{0}$. Then, $f=\bigvee\limits_{\alpha\ge\alpha_{0}} f\wedge f_{\alpha}\le g$.\medskip

Necessity: Fix $\alpha_{0}$ and assume that $e\ge f\wedge f_{\alpha}$, for every $\alpha\ge\alpha_{0}$. Let $H=\left(f_{\alpha}\right)_{\alpha\in A}^{\nearrow}$, which is bounded below and nonempty. Let $h\in H$, and let $\alpha_{1}$ be such that $h\ge f_{\alpha}$, for every $\alpha\ge\alpha_{1}$; since $h\ge f$ we in fact have $h\ge f\vee f_{\alpha}$. Take $\alpha_{2}\ge\alpha_{0},\alpha_{1}$. For every $\alpha\ge\alpha_{2}$ we have that $h+f\wedge e\ge f\vee f_{\alpha}+ f\wedge f_{\alpha}=f+f_{\alpha}$, from where $h-\left(f-e\right)^{+}=h-f+f\wedge e\ge f_{\alpha}$. Hence, $h-\left(f-e\right)^{+}\in H$, from where $H-\left(f-e\right)^{+}\subset H$, and so according to Lemma \ref{arch} $\left(f-e\right)^{+}=0_{F}$, which yields $f\le e$. Thus, $\bigvee\limits_{\alpha\ge\alpha_{0}} f\wedge f_{\alpha}=f$.
\end{proof}

In particular, $0_{F}\le f_{\alpha}\xrightarrow[]{o}0_{F}$ iff $\bigwedge\left(f_{\alpha}\right)_{\alpha\in A}^{\nearrow}=0_{F}$, and in this event $\bigwedge\left\{f_{\alpha}\right\}_{\alpha\in A}=0_{F}$.

\begin{proposition}\label{ocls5}
Order convergence is idempotent.
\end{proposition}
\begin{proof}
Assume that for $\left(f_{\alpha}\right)_{\alpha\in A}\subset F_{+}$ there is $0_{F}\le g_{\beta}\xrightarrow[]{o}0_{F}$ such that $\left(f_{\alpha}-g_{\beta}\right)^{+}\xrightarrow[\alpha]{o}0_{F}$, for every $\beta$. It follows from the comment above that for every $\beta$ we have \linebreak $\bigwedge\left(\left(f_{\alpha}-g_{\beta}\right)^{+}\right)_{\alpha\in A}^{\nearrow}=0_{F}$. Note that $g_{\beta}+\left(\left(f_{\alpha}-g_{\beta}\right)^{+}\right)_{\alpha\in A}^{\nearrow}\subset \left(f_{\alpha}\right)_{\alpha\in A}^{\nearrow}$, and so every lower bound of $\left(f_{\alpha}\right)_{\alpha\in A}^{\nearrow}$ has to be below $g_{\beta}$, for every $\beta$. Since $\bigwedge\left\{g_{\beta}\right\}_{\beta\in B}=0_{F}$, it follows that $\bigwedge\left(f_{\alpha}\right)_{\alpha\in A}^{\nearrow}=0_{F}$, and so $f_{\alpha}\xrightarrow[]{o}0_{F}$.
\end{proof}

\section{Order continuity of homomorphisms}\label{ochs}

In this section $F$ is an Archimedean vector lattice. Recall that an order closed ideal is called a \emph{band}. Intersection of any collection of bands is a band. For any $G\subset F$ its disjoint complement $G^{d}$ is a band, and conversely, $E\subset F$ is a band iff $E=E^{dd}$. In fact, $G^{dd}$ is the band generated by $G\subset F$ (see \cite[Theorem 1.39]{ab}). Note that $\left(G\cup G^{d}\right)^{d}=\left\{0_{F}\right\}$, from where $\left(G\cup G^{d}\right)^{dd}=F$. If $e\in F_{+}$ is such that $\left\{e\right\}^{dd}=F$ (or equivalently $\left\{e\right\}^{d}=\left\{0_{F}\right\}$), we will call it a \emph{weak unit}.\medskip

We now start studying the \emph{unbounded order (uo)} convergence, i.e. the unbounded modification of the order convergence. It is easy to see that uo convergence is weaker than the order convergence, but is still Hausdorff, and of course it is locally solid linear. Also, a net order converges iff it uo converges and is eventually order bounded. We get the following corollary of Proposition \ref{umb} and Proposition \ref{ocls5}.

\begin{corollary}[cf. \cite{lc}]\label{lc}
For a net $\left(f_{\alpha}\right)_{\alpha\in A}\subset F_{+}$ the set $H=\left\{h\in F:~ \left|h\right|\wedge f_{\alpha}\xrightarrow[]{o} 0_{F}\right\}$ is a band. If $e\in F_{+}$ is a weak unit in $F$, then $f_{\alpha}\xrightarrow[]{uo} 0_{F}$ iff $e \wedge f_{\alpha}\xrightarrow[]{o} 0_{F}$.
\end{corollary}

Recall that an operator $T:F\to E$ between vector lattices is called a \emph{homomorphism} if it preserves the lattice operations (equivalently if $T\left|f\right|=\left|Tf\right|$, for every $f\in F$). It is easy to see that every homomorphism is positive. The equivalence of the conditions (iii) and (v) in the following theorem were very recently proven in \cite{tag} using a different method. For other equivalent conditions see \cite[Theorem 3.10]{erz}.

\begin{theorem}\label{huo}For a homomorphism $T:F\to E$ the following conditions are equivalent:
\item[(i)] $T$ maps nets decreasing to $0_{F}$ into nets decreasing to $0_{E}$;
\item[(ii)] $T$ preserves the extremums of arbitrary sets (i.e. if $G\subset F$ is such that $g=\bigvee G$, then $Tg=\bigvee TG$, and the same for the infimum);
\item[(iii)] $T$ is order continuous;
\item[(iv)] There is an ideal $H\subset F$ such that $\left.T\right|_{H}$ is order continuous and $TF\subset \left(TH\right)^{dd}$;
\item[(v)] $T$ is uo-continuous.
\end{theorem}
\begin{proof}
For (i)$\Rightarrow$(ii) observe that for any $G\subset F$ we have $\bigvee G^{\vee}=\bigvee G$, and $G^{\vee}$ is an increasing net. (i)$\Rightarrow$(iii) follows from Proposition \ref{idcont}. The converse implications are trivial in both cases, as is (iii)$\Rightarrow$(iv). (v)$\Rightarrow$(i) follows from the fact that for a decreasing net in $F$ to decrease to $0_{F}$ is equivalent to uo converge to $0_{F}$.\medskip

(iv)$\Rightarrow$(v): Let $\left(f_{\alpha}\right)_{\alpha\in A}\subset F_{+}$ be uo-null. Let $e\in \left(TH\right)_{+}$, so that there is $h\in H_{+}$ with $Th=e$. Since $\left(f_{\alpha}\right)_{\alpha\in A}$ is uo-null, $H\ni h\wedge f_{\alpha}\xrightarrow[]{o}0_{F}$, and from the order continuity of $\left.T\right|_{H}$, we have $e\wedge Tf_{\alpha}=T\left(h\wedge f_{\alpha}\right)\xrightarrow[]{o}0_{E}$. If $e\in \left(TF\right)_{+}^{d}$, we have that $e\wedge Tf_{\alpha}=0_{E}$. According to Corollary \ref{lc}, $G=\left\{e\in E:~ \left|e\right|\wedge Tf_{\alpha}\xrightarrow[]{o} 0_{F}\right\}$ is a band, which contains both $\left(TH\right)_{+}$ and $\left(TF\right)_{+}^{d}=\left(TH\right)_{+}^{d}$, therefore it is equal to $E$. Thus, $Tf_{\alpha}\xrightarrow[]{uo}0_{F}$, and so $T$ is uo-continuous.
\end{proof}

A sublattice $E\subset F$ is called \emph{regular} if for every $G\subset E$ such that $\bigwedge_{E}G=0_{F}$ (i.e. the infimum within the order structure of $E$) we have $\bigwedge_{F}G=0_{F}$. It is easy to see that every ideal is regular. The following was proven in \cite{pap1} and rediscovered in \cite{as} and \cite{gtx}.

\begin{theorem}[Papangelou]\label{gt1}Let $F$ be Archimedean. For a sublattice $E\subset F$ the following conditions are equivalent:
\item[(i)] $E$ is regular;
\item[(ii)] The inclusion operator from $E$ into $F$ is order continuous;
\item[(iii)] If $\left(f_{\alpha}\right)_{\alpha\in A}\subset E_{+}$ is order bounded in $E$, then $f_{\alpha}\xrightarrow[]{o} 0_{F}$ in $E$ iff $f_{\alpha}\xrightarrow[]{o} 0_{F}$ in $F$;
\item[(iv)] The inclusion from $\left[0_{E},e\right]$ into $F$ is an order embedding, for every $e\in E_{+}$;
\item[(v)] The inclusion operator from $E$ into $F$ is uo-continuous;
\item[(vi)] The inclusion operator from $E$ into $F$ is a uo-embedding.
\end{theorem}
\begin{proof}
(i)$\Leftrightarrow$(ii)$\Leftrightarrow$(v) follows from Theorem \ref{huo}, and (iv) is a rephrasing of (iii). (vi)$\Rightarrow$(v) and (iii)$\Rightarrow$(ii) are trivial.\medskip

(i)$\Rightarrow$(iii): Since (i)$\Rightarrow$(ii) is already established we have that $E_{+}\ni f_{\alpha}\xrightarrow[]{o} 0_{F}$ in $E$ implies $f_{\alpha}\xrightarrow[]{o} 0_{F}$ in $F$, and so we only need to prove the converse. Let $\left(f_{\alpha}\right)_{\alpha\in A}^{\nearrow_{F}}$ and $\left(f_{\alpha}\right)_{\alpha\in A}^{\nearrow_{E}}$ stand for the collections of eventual upper bounds of $\left(f_{\alpha}\right)_{\alpha\in A}$ in $F$ and $E$, respectively. We know that $\bigwedge\left(f_{\alpha}\right)_{\alpha\in A}^{\nearrow_{F}}=0_{F}$, and need to show that $\bigwedge\left(f_{\alpha}\right)_{\alpha\in A}^{\nearrow_{E}}=0_{F}$. Assume that $\left(f_{\alpha}\right)_{\alpha\in A}^{\nearrow_{E}}\ge g\in E$, and so according to part (ii) of Proposition \ref{ocs}, $\bigvee\limits_{\alpha\ge\alpha_{0}} g\wedge f_{\alpha}=g$, for every $\alpha_{0}$. Note that the infimum is valid in both $E$ and $F$, hence, using part (ii) of Proposition \ref{ocs} again, we conclude that $\bigwedge\left(f_{\alpha}\right)_{\alpha\in A}^{\nearrow_{F}}\ge g$, and so $g \le 0_{F}$. Thus, $\bigwedge\left(f_{\alpha}\right)_{\alpha\in A}^{\nearrow_{E}}=0_{F}$.\medskip

Now that the equivalence of the conditions (i)-(v) are established it is left to prove that if under these conditions $\left(f_{\alpha}\right)_{\alpha\in A}\subset E_{+}$ is uo-null in $F$, then it is uo-null in $E$. Indeed, for every $e\in E_{+}$ the net $\left(e\wedge f_{\alpha}\right)_{\alpha\in A}\subset\left[0_{F},e\right]\cap F$ is order null in $F$. It follows that $e\wedge f_{\alpha}\xrightarrow[]{o} 0_{F}$ in $E$. As $e$ was arbitrary, we conclude that $f_{\alpha}\xrightarrow[]{uo} 0_{F}$ in $E$.
\end{proof}

Let us now investigate when a sublattice is order embedded. Clearly, regularity is a necessary condition. Recall that an ideal $H\subset F$ is called a \emph{projection band} if $F=H+H^{d}$ (and so there is a projection $P:F\to H$, which is a homomorphism with $H=PF$). We will need the following auxiliary claim:

\begin{lemma}\label{oh}Let $H\subset F$ be an ideal, such that every $G\subset H$ which is order bounded in $F$ is order bounded in $H$. Then, $H$ is a projection band.
\end{lemma}
\begin{proof}
Take $f\in F_{+}$. The set $H\cap \left[0_{F},f\right]$ is order bounded in $F$, and so there is $h\in H$ such that $H\cap \left[0_{F},f\right]\le h$. Then, $h\wedge f\in H\cap \left[0_{F},f\right]$ satisfies $H\cap \left[0_{F},f\right]\le h\wedge f$, and so $h\wedge f=\bigvee \left(H\cap \left[0_{F},f\right]\right)$. Let us show that $\left(f-h\right)^{+}\bot H$. Indeed, if $g\in\left[0_{F},\left(f-h\right)^{+}\right]\cap H$, then $g+f\wedge h\in H$ and $g+f\wedge h\le \left(f-h\right)^{+}+f\wedge h=f$ imply that $g+f\wedge h\in H\cap \left[0_{F},f\right]\le f\wedge h$, and so $g=0_{F}$. Thus, $f=f\wedge h+\left(f-h\right)^{+}\in H+H^{d}$. Since $f$ was chosen arbitrarily, we conclude that $H$ is a projection band.
\end{proof}

It is easy to see that the converse is also true. Moreover, it is enough to only consider $G$'s which consist of mutually disjoint elements (see \cite{erz0}).

A sublattice $E\subset F$ is called \emph{order dense} if $E\cap \left(0_{F},f\right]\ne\varnothing$, for every $f\in F_{+}\backslash \left\{0_{F}\right\}$. This condition is equivalent to the fact that $f=\bigvee\left(E\cap \left[0_{F},f\right]\right)$, for every $f\in F_{+}$ (see \cite[Theorem 1.34]{ab}), and implies regularity. A regular / order dense sublattice of a regular / order dense sublattice is regular / order dense. We now present a partial characterization of order embedded sublattices.

\begin{proposition}\label{oh2}
If $E\subset F$ is a regular sublattice, such that $I\left(E\right)$ is a projection band, then $E$ is order embedded into $F$. Conversely, if $E$ is order dense in $I\left(E\right)$ and order embedded, then $I\left(E\right)$ is a projection band.
\end{proposition}
\begin{proof}
We start with the first claim. If a net $\left(e_{\alpha}\right)_{\alpha\in A}\subset E$ is order null in $F$, it is eventually order bounded in $F$. Let $f\in F_{+}$ and $\alpha_{0}$ be such that $\left\{e_{\alpha}\right\}_{\alpha\ge\alpha_{0}}\subset\left[-f,f\right]$. Let $P:F\to I\left(E\right)$ be the band projection. There is $e\in E_{+}$ such that $e\ge Pf$, hence $\left\{e_{\alpha}\right\}_{\alpha\ge\alpha_{0}}\subset\left[-e,e\right]$. Since according to Theorem \ref{gt1}, $\left[-e,e\right]\cap E$ is order embedded into $F$, it follows that  $\left(e_{\alpha}\right)_{\alpha\in A}$ is order null in $E$.\medskip

Let us prove the second claim. Fix $f\in F_{+}$ and consider $E\cap \left[0_{F},f\right]$ to be an increasing net indexed by itself. This net is order bounded in $F$, and so it is order Cauchy in $F$ (see \cite[Lemma 2.1]{lc}). Hence, this net is order Cauchy in $E$, and therefore eventually order bounded in $E$. It follows that there is $e\in E$ such that $E\cap \left[0_{F},f\right]\le e$.  Take $g\in I\left(E\right)\cap \left[0_{F},f\right]$. Since $E$ is order dense in $I\left(E\right)$, we have that $g=\bigvee \left(E\cap \left[0_{F},g\right]\right)$. As $E\cap \left[0_{F},g\right]\subset E\cap \left[0_{F},f\right]\le e$, it follows that $g\le e$. Hence, $I\left(E\right)\cap \left[0_{F},f\right]\le e$, and so according to Lemma \ref{oh}, $I\left(E\right)$ is a projection band.
\end{proof}

It seems hopeless to find a better characterization of order embeddedness, since any finitely dimensional sublattices are order embedded. However, for the ideals the situation is simple.

\begin{corollary}An ideal $H\subset E$ is a projection band iff inclusion operator from $H$ into $F$ is an order embedding.
\end{corollary}

\section{Characterizations of uo convergence}\label{cuos}

In this section we add some characterizations of the unbounded order convergence, which will liberate it from the ``derivative role'' with respect to the order convergence. Everywhere in the section $F$ is an Archimedean vector lattice. The following lemma will be utilized.

\begin{lemma}\label{arch2}If $f,e\in F_{+}$ are such that $\left(f-ne\right)^{+}\ge e$, for every $n\in\N$, then $e=0_{F}$.\end{lemma}
\begin{proof}
Let $f_{n}=\left(f-ne\right)^{+}$; we have that $f_{n}-e\ge 0_{F}$ and $f_{n}-e\ge f-\left(n+1\right)e$, from where $f_{n}-e\ge f_{n+1}\ge 0_{F}$, and so $f_{n}\ge f_{n+1}+e\ge e$, for every $n\in\N$. Assume that we have proven that $f_{n}\ge me$, for some $m\in\N$, and every $n\in\N$. Then, for every $n\in\N$ we have $f_{n}\ge f_{n+1}+e\ge \left(m+1\right)e$. Hence, by induction $f_{n}\ge me$, for every $m,n\in\N$, and so $e=0_{F}$.
\end{proof}

Recall that $F$ is said to have a \emph{principal projection property (PPP)} if $\left\{e\right\}^{dd}$ is a projection band, for every $e\in F$; the corresponding band projection is denoted by $P_{e}$. We also call $h$ a \emph{component} of $e$ if $h\bot e-h$. The condition (ii) in the following result appeared in \cite{pap2}, while the condition (iv) is a modification of a condition in \cite{et}.

\begin{theorem}\label{uo}For $\left(f_{\alpha}\right)_{\alpha\in A}\subset F_{+}$ the following conditions are equivalent:
\item[(i)] $f_{\alpha}\xrightarrow[]{uo}0_{F}$;
\item[(ii)] If $h\in F_{+}$ is such that $\bigvee\limits_{\alpha\ge\alpha_{0}}h\wedge f_{\alpha}=h$, for every $\alpha_{0}$, then $h=0_{F}$;
\item[(iii)] For every $h>0_{F}$ there are $e\in\left(0_{F},h\right]$ and $\alpha_{0}$ such that $\left(f_{\alpha}-h\right)^{+}\bot e$, for every $\alpha\ge\alpha_{0}$.\medskip

Moreover, if $F$ has PPP, the conditions above are equivalent to:
\item[(iv)] For every $h>0_{F}$ there is a component $e$ of $h$ and $\alpha_{0}$ such that $P_{e}f_{\alpha}\le e$, for every $\alpha\ge\alpha_{0}$.
\end{theorem}
\begin{proof}
First, observe that (ii) is equivalent to the fact that for every $h>0_{F}$ there are $g\in \left[0_{F},h\right)$ and $\alpha_{0}$ such that $h\wedge f_{\alpha}\le g$, for every $\alpha\ge\alpha_{0}$ (because this precisely means that $\bigvee\limits_{\alpha\ge\alpha_{0}}h\wedge f_{\alpha}\ne h$).\medskip

(i)$\Rightarrow$(ii): Take $h>0_{F}$. Since $h\wedge f_{\alpha}\xrightarrow[]{o}0_{F}$, it follows that $\bigwedge \left(h\wedge f_{\alpha}\right)_{\alpha\in A}^{\nearrow}=0_{F}$. Hence, there are $e\not\ge h$ and $\alpha_{0}$ such that $h\wedge f_{\alpha}\le e$, for every $\alpha\ge\alpha_{0}$. Then, $g=e\wedge h$ satisfies the requirements.\medskip

(ii)$\Rightarrow$(iii): Take $e=h-g\in\left(0_{F},h\right]$. Then, for $\alpha\ge\alpha_{0}$ we have $e\le h-h\wedge f_{\alpha}=\left(h-f_{\alpha}\right)^{+}\bot \left(f_{\alpha}-h\right)^{+}$.\medskip

(iii)$\Rightarrow$(i): First, assume that $\left(f_{\alpha}\right)_{\alpha\in A}$ is order bounded, and so $G=\left(f_{\alpha}\right)_{\alpha\in A}^{\nearrow}\ne\varnothing$. Assume that $G\ge h>0_{F}$. Take any $g\in G$ and $e\in\left(0_{F},\frac{h}{2}\right]$ and find $\alpha_{0}$ such that $\left(f_{\alpha}-\frac{h}{2}\right)^{+}\bot e$, for every $\alpha\ge\alpha_{0}$. Let $\alpha_{1}$ be such that $f_{\alpha}\le g$, for every $\alpha\ge\alpha_{1}$. For $\alpha\ge \alpha_{0},\alpha_{1}$ and any $n\in\N$ we have $$f_{\alpha}+ne\le ne+\left(f_{\alpha}-\frac{h}{2}\right)^{+}+\frac{h}{2}=ne\vee\left(f_{\alpha}-\frac{h}{2}\right)+\frac{h}{2}\le ne\vee\left(g-\frac{h}{2}\right)+\frac{h}{2}=g\vee\left(\frac{h}{2}+ne\right),$$ from where $f_{\alpha}\le \left(g-ne\right)\vee \frac{h}{2}$. Hence, $\left(g-ne\right)\vee \frac{h}{2}\in G$, and so $\left(g-ne\right)\vee \frac{h}{2}\ge h$. Therefore, $\left(g-\frac{h}{2}-ne\right)^{+} \ge \frac{h}{2}\ge e$, for every $n$, and so $e=0_{F}$, according to Lemma \ref{arch2}. We have reached a contradiction, and so $\bigwedge G=0_{F}$. Consequently, $f_{\alpha}\xrightarrow[]{o}0_{F}$.\medskip

Now, in the general case, for every $f\in F_{+}$ the net $\left(f\wedge f_{\alpha}\right)_{\alpha\in A}$ satisfies the conditions of the special case, and so $f\wedge f_{\alpha}\xrightarrow[]{o}0_{F}$. Thus, $f_{\alpha}\xrightarrow[]{uo}0_{F}$.\medskip

It is easy to see that (iv)$\Rightarrow$(iii), and so it is left to show the converse under the assumption that $F$ has PPP. Let $h>0_{F}$, and let $\alpha_{0}$ and $e$ be produced by (iii). Since $h\ge e$ we have $g=P_{e}h\ge e$. Note that $g$ is a component of $h$. On the other hand, $\left\{g\right\}^{dd}\subset \left\{e\right\}^{dd}$, and so $P_{g}=P_{e}$. From $\left(f_{\alpha}-h\right)^{+}\bot e$ it follows that $P_{g}\left(f_{\alpha}-h\right)=P_{e}\left(f_{\alpha}-h\right)\le 0_{F}$, and so $P_{g}f_{\alpha}\le P_{g}h=g$.
\end{proof}

Let us present an alternative proof of (iii)$\Rightarrow$(i). It utilizes the following lemma.

\begin{lemma}\label{inf}
Let $G\subset F_{+}$. Then, $\bigwedge G=0_{F}$ iff for every $h>0_{F}$ there are $e\in\left(0_{F},h\right]$ and $g\in G$ such that $\left(g-h\right)^{+}\bot e$.
\end{lemma}
\begin{proof}
Necessity: If $h>0_{F}$, there is $g\in G$ such that $g\not\ge h$, and so $e=\left(h-g\right)^{+}\in \left(0_{F},h\right]$ and  $\left(g-h\right)^{+}\bot e$.\medskip

Sufficiency: Assume that $G\ge h>0_{F}$. There are $e\in\left(0_{F},\frac{h}{2}\right]$ and $g\in G$ such that $\left(g-\frac{h}{2}\right)^{+}\bot e$. But this contradicts $g\ge h\ge 2e>0_{F}$.
\end{proof}

Recall that a vector lattice is called \emph{order complete} if every order bounded set has a supremum and infimum. There is a unique vector lattice $F^{\delta}$, which is called the \emph{order completion}, such that $F$ embeds into $F^{\delta}$ as an order dense sublattice with $I\left(F\right)=F^{\delta}$ (see \cite[Theorem 32.5]{zl}).

\begin{proof}[An alternative proof of (iii)$\Rightarrow$(i) in Theorem \ref{uo}]
First, assume that $F$ is order complete and $\left(f_{\alpha}\right)_{\alpha\in A}$ is order bounded. For every $h>0_{F}$ there are $e\in\left(0_{F},h\right]$ and $\alpha_{0}$ such that $\left(f_{\alpha}-h\right)^{+}\bot e$, for every $\alpha\ge\alpha_{0}$. Since $\left(f_{\alpha}\right)_{\alpha\in A}$ is order bounded and $F$ is order complete, $g=\bigvee\limits_{\alpha\ge\alpha_{0}}f_{\alpha}$ exists, is an element of $\left(f_{\alpha}\right)_{\alpha\in A}^{\nearrow}$, and satisfies $\left(g-h\right)^{+}\bot e$. Hence, according to Lemma \ref{inf}, $\bigwedge \left(f_{\alpha}\right)_{\alpha\in A}^{\nearrow}=0_{F}$, and so $f_{\alpha}\xrightarrow[]{o}0_{F}$.\medskip

Let us now consider the general case. Let $f\in F_{+}$; the net $\left(f\wedge f_{\alpha}\right)_{\alpha\in A}$ is order bounded in $F$. If $h\in F^{\delta}_{+}\backslash\left\{0_{F^{\delta}}\right\}$, then since $F$ is order dense in $F^{\delta}$, there is $g\in \left(0_{F^{\delta}},h\right]\cap F$. Therefore, there is $e\in\left(0_{F},g\right]\subset \left(0_{F^{\delta}},h\right]$ such that $\left(f_{\alpha}-h\right)^{+}\le \left(f_{\alpha}-g\right)^{+}\bot e$, for every $\alpha\ge\alpha_{0}$. Hence, $\left(f\wedge f_{\alpha}\right)_{\alpha\in A}$ satisfies the condition of the special case and so $f\wedge f_{\alpha}\xrightarrow[]{o}0_{F^{\delta}}$ in $F^{\delta}$. According to Theorem \ref{gt1}, it follows that $f\wedge f_{\alpha}\xrightarrow[]{o}0_{F}$ in $F$, and since $f$ was arbitrary, $f_{\alpha}\xrightarrow[]{uo}0_{F}$.
\end{proof}

We would like to present another characterization of the uo convergence originating from \cite{pap2}. We will call $f\in F$ an \emph{interelement} of a net $\left(f_{\alpha}\right)_{\alpha\in A}\subset F$ if $\varnothing\ne\left(f_{\alpha}\right)_{\alpha\in A}^{\searrow}\le f\le \left(f_{\alpha}\right)_{\alpha\in A}^{\nearrow}\ne\varnothing$, and a \emph{weak interelement} if $e\vee f\wedge h$ is an interelement of $\left(e\vee f_{\alpha}\wedge h\right)_{\alpha\in A}$, for any $e\le h$. It is easy to see that $f$ is an interelement of the net iff it is a weak interelement and the net is eventually order bounded. Clearly, $f$ is a (weak) interelement of $\left(f_{\alpha}\right)_{\alpha\in A}$ iff $0_{F}$ is a (weak) interelement of $\left(f_{\alpha}-f\right)_{\alpha\in A}$ and iff $0_{F}$ is a (weak) interelement of $\left(f-f_{\alpha}\right)_{\alpha\in A}$. According to part (i) of Proposition \ref{ocs}, $f_{\alpha}\xrightarrow[]{o}f$ if and only if $f$ is the unique interelement of $\left(f_{\alpha}\right)_{\alpha\in A}$. It turns out that a similar fact holds for the uo convergence.

\begin{theorem}[Papangelou]\label{uo2}Let $\left(f_{\alpha}\right)_{\alpha\in A}\subset F$ and $f\in F$. Then:
\item[(i)] $f$ is a weak interelement if and only if $\bigvee\limits_{\alpha\ge\alpha_{0}} \left(f\wedge f_{\alpha}\right)=f=\bigwedge\limits_{\alpha\ge\alpha_{0}} \left(f\vee f_{\alpha}\right)$, for every $\alpha\ge\alpha_{0}$.
\item[(ii)] $f_{\alpha}\xrightarrow[]{uo}f$  if and only if $f$ is the unique weak interelement of $\left(f_{\alpha}\right)_{\alpha\in A}$.
\end{theorem}
\begin{proof}
(i): Sufficiency: It is easy to see that if $f=\bigvee\limits_{\alpha\ge\alpha_{0}} \left(f\wedge f_{\alpha}\right)$, then $e\vee f\wedge h=\bigvee\limits_{\alpha\ge\alpha_{0}} \left(\left(e\vee f\wedge h\right)\wedge\left(e\vee f_{\alpha}\wedge h\right)\right)$, for $e\le h$, from where according to part (ii) of Proposition \ref{ocs}, $e\vee f\wedge h\le \left(e\vee f_{\alpha}\wedge h\right)_{\alpha\in A}^{\nearrow}$. The case of the eventual lower bounds is analogous, and so $e\vee f\wedge h$ is an interelement of $\left(e\vee f_{\alpha}\wedge h\right)_{\alpha\in A}$.\medskip

Necessity: Fix $\alpha_{0}$ and assume that $e\le f\vee f_{\alpha}$, for every $\alpha\ge\alpha_{0}$. Then, $e$ is an eventual lower bound of the net $\left(f\vee f_{\alpha}\wedge \left(f\vee e\right)\right)_{\alpha\in A}$. Since $f$ is an interelement of this net, it follows that $f\ge e$, and so $f=\bigwedge\limits_{\alpha\ge\alpha_{0}} \left(f\vee f_{\alpha}\right)$. The other equality is proven analogously.

(ii): Necessity: If $f_{\alpha}\xrightarrow[]{uo}f$, then according to part (iv) of Proposition \ref{unba}, for every $e\le h$, we have that $e\vee f_{\alpha}\wedge h\xrightarrow[]{o}e\vee f\wedge h$; then $e\vee f\wedge h$ is the only interelement of $\left(e\vee f_{\alpha}\wedge h\right)_{\alpha\in A}$. Hence, in particular $f$ is a weak interelement, and for any weak interelement $g$ taking $e=g\wedge f$ and $h=g\vee f$ yields $g=f$.\medskip

Sufficiency: It is enough to assume that $f=0_{F}$. We will use the characterization of the weak interelements given in (i). In particular, $\bigwedge\limits_{\alpha\ge\alpha_{0}}f^{+}_{\alpha}=0_{F}=\bigwedge\limits_{\alpha\ge\alpha_{0}}f^{-}_{\alpha}$, for every $\alpha_{0}$. Let $h\ge 0_{F}$ and assume that $\bigvee\limits_{\alpha\ge\alpha_{0}}h\wedge f_{\alpha}^{+}=h$, for every $\alpha_{0}$. For every $\alpha_{0}$ we have $\bigwedge\limits_{\alpha\ge\alpha_{0}}\left(h\vee f_{\alpha}\right)=\bigwedge\limits_{\alpha\ge\alpha_{0}}\left(h\vee f_{\alpha}^{+}\right)=h\vee \bigwedge\limits_{\alpha\ge\alpha_{0}}f_{\alpha}^{+}=h$, and $$\bigvee\limits_{\alpha\ge\alpha_{0}}\left(h\wedge f_{\alpha}\right)=\bigvee\limits_{\alpha\ge\alpha_{0}}\left(h\wedge f_{\alpha}\right)^{+}-\bigwedge\limits_{\alpha\ge\alpha_{0}}\left(h\wedge f_{\alpha}\right)^{-}=\bigvee\limits_{\alpha\ge\alpha_{0}}h\wedge f_{\alpha}^{+}-\bigwedge\limits_{\alpha\ge\alpha_{0}}f^{-}_{\alpha}=h.$$ Hence, $h$ is a weak interelement of $\left( f_{\alpha}\right)_{\alpha\in A}$, and so $h=0_{F}$. Therefore, the condition (ii) of Theorem \ref{uo} is satisfied, and so $f_{\alpha}^{+}\xrightarrow[]{uo}0_{F}$. Analogously, $f_{\alpha}^{-}\xrightarrow[]{uo}0_{F}$, thus $f_{\alpha}\xrightarrow[]{uo}0_{F}$.
\end{proof}

\section{Uo convergence in spaces of continuous functions}\label{uocsfs}

An application of the new approach to the uo convergence is the characterization of this convergence in the vector lattices of continuous functions, first proven in \cite{et}.

\begin{theorem}\label{couo1}Let $X$ be a Tychonoff space. A net $\left(f_{\alpha}\right)_{\alpha\in A}$ in $\Co\left(X\right)_{+}$ is uo-null if and only if for every open nonempty $U\subset X$ and $\varepsilon>0$ there is an open nonempty $V\subset U$ and $\alpha_{0}$ such that $f_{\alpha}\left(x\right)\le\varepsilon$, for every $x\in V$ and $\alpha\ge\alpha_{0}$.
\end{theorem}
\begin{proof}
Necessity: Let $U\subset X$ be open and nonempty, and let $\varepsilon>0$. Take a continuous function $u\in\left(\0,\1_{U}\right]$ and using Theorem \ref{uo} find a continuous $v\in\left(\0,\varepsilon u\right]$ and $\alpha_{0}$ such that $\left(f_{\alpha}-\varepsilon u\right)^{+}\bot v$, for every $\alpha\ge\alpha_{0}$. Then, $\left(f_{\alpha}-\varepsilon u\right)^{+}$ vanishes on $\varnothing\ne V:=\supp v\subset U$, and so $f_{\alpha}\left(x\right)\le\varepsilon u\left(x\right)\le\varepsilon$, for every $x\in V$ and $\alpha\ge\alpha_{0}$.\medskip

Sufficiency: Let $u>\0$ be continuous and let open nonempty $U$ and $\varepsilon$ be such that $u\ge\varepsilon\1_{U}$. Let $V\subset U$ be open nonempty and such that $f_{\alpha}\left(x\right)\le\varepsilon$, for every $x\in V$ and $\alpha\ge\alpha_{0}$, for some $\alpha_{0}$. Let $v\in\left(\0,\varepsilon\1_{V}\right]\subset\left(\0,u\right]$ be continuous. It is easy to see that $\left(f_{\alpha}-u\right)^{+}$ vanishes on $V$, and so is disjoint with $v$, for every $\alpha\ge\alpha_{0}$. Hence, according to Theorem \ref{uo}, $f_{\alpha}\xrightarrow[]{uo}\0$.
\end{proof}

The following answers a question of Mark Roelands.

\begin{corollary}\label{couo3}
The correspondence $f\mapsto f^{-1}$ is uo-continuous on the set of invertible elements of $\Co\left(X\right)_{+}$.
\end{corollary}
\begin{proof}
Denote the set in question by $\Co\left(X\right)_{++}$. Let $\left(f_{\alpha}\right)_{\alpha\in A}$ and $f$ be in $\Co\left(X\right)_{++}$. Assume that $f_{\alpha}\xrightarrow[]{uo}f$; we need to show that $f_{\alpha}^{-1}\xrightarrow[]{uo}f^{-1}$. Let $U\subset X$ be open and nonempty and let $\varepsilon>0$. Since $f$ is invertible there is an open $\varnothing\ne V\subset U$ and $r>0$ such that $f\ge r\1_{V}$. Let $\delta =\frac{r^{2}\varepsilon}{1+r\epsilon}$. There is an open $\varnothing\ne W\subset V$ and $\alpha_{0}$ such that $\left|f_{\alpha}\left(x\right)-f\left(x\right)\right|\le\delta$, for every $x\in W$ and $\alpha\ge\alpha_{0}$. We then have on $W$ that $\left|f_{\alpha}^{-1}\left(x\right)-f^{-1}\left(x\right)\right|=\frac{\left|f_{\alpha}\left(x\right)-f\left(x\right)\right|}{\left|f_{\alpha}\left(x\right)\right|\left|f\left(x\right)\right|}\le \frac{\delta}{r\left(r-\delta\right)}=\varepsilon$. Since $U$ and $\varepsilon$ were chosen arbitrarily, $f_{\alpha}^{-1}\xrightarrow[]{uo}f^{-1}$.
\end{proof}

\begin{remark}
Recall that if $X$ is extremally disconnected, the  continuous $\left[-\8,\8\right]$-valued functions which are finite apart of a nowhere dense set form a vector lattice, which we denote $\Co^{\8}\left(X\right)$. For more information about this object see \cite[Section 7.2]{ab0}. Both Theorem \ref{couo1} and Corollary \ref{couo3} are still valid for $\Co^{\8}\left(X\right)$ with the same proofs going through. Similarly, for any continuous function $\varphi:D\to \left[-\8,\8\right]$, where $D\subset\left[-\8,\8\right]$, the correspondence $f\mapsto \varphi\circ f$ is uo-continuous on the appropriately chosen subset of $\Co\left(X\right)$ or $\Co^{\8}\left(X\right)$.
\qed\end{remark}

We will call a topological space $X$ \emph{almost locally compact} if the set of all points of $X$ that have a compact neighborhood is dense.

\begin{theorem}\label{couo2}For a Tychonoff space $X$ TFAE:
\item[(i)] $X$ is almost locally compact;
\item[(ii)] $X$ contains a dense locally compact subspace;
\item[(iii)] Every open nonempty $U\subset X$ contains an open nonempty relatively compact subset;
\item[(iv)] The compact open convergence on $\Co\left(X\right)$ is stronger than the uo convergence.
\end{theorem}
\begin{proof}
(i)$\Rightarrow$(ii) is trivial, while the converse follows from the fact that a dense locally compact subspace is always open (see \cite[3.15 (d)]{gj}). (i)$\Leftrightarrow$(iii) is easy to see, and (iii)$\Rightarrow$(iv) follows from Theorem \ref{couo1}.\medskip

(iv)$\Rightarrow$(iii): Let $U$ be an open nonempty subset of $X$ which contains no nonempty relatively compact open sets. Consider the net $\left(f_{K,O}\right)_{K,O}$, where $K\subset X$ is compact, $O$ is an open neighborhood of $K$, the net is ordered by inclusion in the first variable, and $f_{K,O}$ vanishes on $K$ and is equal to $1$ outside of $O$. Clearly, $\left(f_{K,O}\right)_{K,O}$ converges to $\0$ in the compact-open topology. Let us show that it is not uo-null. According to Theorem \ref{couo1}, it is enough to show that for any open nonempty $V\subset U$ and any index $\left(K,O\right)$ there is an index $\left(K,W\right)$ and $x\in V$ such that  $f_{K,W}\left(x\right)=1$. In order to construct these $x$ and $W$ observe that our assumption about $U$ implies that $V$ is not relatively compact, and so $V\not\subset K$. Take $x\in V\backslash K$ and let $W$ be an open neighborhood of $K$ such that $x\notin \overline{W}$. Then, $f_{K,W}\left(x\right)=1$, as required.
\end{proof}

\section{Order continuity of positive operators}\label{ucpos}

In this section $T:F\to E$ is a positive operator between Archimedean vector lattices endowed with locally solid Hausdorff additive convergences. We start with a generalization of \cite[Proposition 3.1]{erz}.

\begin{proposition}\label{loc0}
Let  $H=\left\{h\in F:~ \left.T\right|_{F_{\left|h\right|}}\mbox{ is order continuous}\right\}$. Then $H$ is an ideal and $\left.T\right|_{H}$ is order continuous. Moreover, if $T$ is continuous, then $H$ is closed.
\end{proposition}
\begin{proof}
The fact that $H$ is an ideal on which $T$ is order continuous was proven in \cite[Proposition 3.1]{erz}. Assume that $T$ is continuous. Let $h\in \overline{H}_{+}^{1}$, and so there is a net $\left(h_{\alpha}\right)_{\alpha \in A}\subset \left[0_{F},h\right]$ such that $h_{\alpha}\to h$. Assume that $\left(f_{\gamma}\right)_{\gamma \in \Gamma}$ is such that $h\ge f_{\gamma}\downarrow 0_{F}$ and $0_{E}\le e\le Tf_{\gamma}$, for every $\gamma$. Then, $h_{\alpha}\ge f_{\gamma}\wedge h_{\alpha}\downarrow_{\gamma} 0_{F}$, for every $\alpha$, and $f_{\gamma}-f_{\gamma}\wedge h_{\alpha}=\left(f_{\gamma}-h_{\alpha}\right)^{+}\le h-h_{\alpha}$, for every $\gamma$ and $\alpha$. Hence, $$e\le Tf_{\gamma}= T\left(f_{\gamma}\wedge h_{\alpha}\right)+T\left(f_{\gamma}-f_{\gamma}\wedge h_{\alpha}\right)\le T\left(f_{\gamma}\wedge h_{\alpha}\right)+T\left(h-h_{\alpha}\right),$$ for every $\gamma$ and $\alpha$. Since $h_{\alpha}\in H$ and $f_{\gamma}\wedge h_{\alpha}\downarrow_{\gamma} 0_{F}$, we have $T\left(f_{\gamma}\wedge h_{\alpha}\right)\downarrow_{\gamma} 0_{F}$, from where $0_{E}\le e\le T\left(h-h_{\alpha}\right)$, for every $\alpha$. Since $T$ is continuous, $h_{\alpha}\to h$, and $e+E_{+}$ is closed, we conclude that $e=0_{E}$. Thus, $h\in H$.
\end{proof}

We will call $G\subset F$ \emph{majorizing} if $I\left(G\right)=F$ and \emph{topologically majorizing} if $\overline{I\left(G\right)}=F$. The following is a generalization of \cite[Corollary 3.2]{erz}.

\begin{proposition}\label{oc}For a positive continuous $T:F\to E$ the following conditions are equivalent:
\item[(i)] $T$ is order continuous;
\item[(ii)] There is a topologically majorizing $G\subset F_{+}$ such that $\left.T\right|_{F_{g}}$ is order continuous, for every $g\in G$;
\item[(iii)] There is a topologically dense ideal $H$ such that $\left.T\right|_{H}$ is order continuous;
\item[(iv)] There is an order dense and topologically majorizing sublattice $H$ such that $\left.T\right|_{H}$ is order continuous.
\end{proposition}
\begin{proof}
(ii)$\Rightarrow$(i) follows from Proposition \ref{loc0}. (i)$\Rightarrow$(iv) is trivial.\medskip

(iv)$\Rightarrow$(iii): If $H$ is order dense and topologically majorizing, then it is order dense and majorizing in $I\left(H\right)$, and $I\left(H\right)$ is a topologically dense ideal. Since order continuity of a positive operator can be propagated from an order dense majorizing sublattice of an Archimedean vector lattice (see \cite[Corollary 3.2]{erz}), we conclude that $\left.T\right|_{I\left(H\right)}$ is order continuous.\medskip

(iii)$\Rightarrow$(ii): If $H$ is a topologically dense ideal and $\left.T\right|_{H}$ is order continuous, then since $F_{h}$ is a regular sublattice of $H$, it follows that $\left.T\right|_{F_{h}}$ is order continuous, for every $h\in H_{+}$, which is a topologically majorizing set.
\end{proof}

\section{The closure of a regular sublattice}\label{crss}

Throughout this section $F$ is again an Archimedean vector lattice. The following result refines \cite[Corollary 2.13]{gtx} and \cite[Theorem 2.13]{gl}.

\begin{theorem}\label{rod}
If $E$ is a regular sublattice of an Archimedean vector lattice $F$, its order adherence is the largest regular sublattice of $F$ in which $E$ is order dense.
\end{theorem}
\begin{proof}
We will denote the order and uo adherences of $E$ by $\overline{E}_{o}^{1}$ and $\overline{E}_{uo}^{1}$ respectively. Since uo convergence is weaker than order convergence, we have $\overline{E}_{o}^{1}\subset \overline{E}_{uo}^{1}$. Next we show that $E$ is order dense in $\overline{E}_{uo}^{1}$.\medskip

Let $f\in \overline{E}_{uo}^{1}$ with $f>0_{F}$. There is a net $\left(e_{\alpha}\right)_{\alpha\in A}\subset E_{+}$ such that $e_{\alpha}\xrightarrow[]{uo}f$, and so there is $G\subset F$ with $\bigwedge G=0_{F}$ which dominates the tails of $\left(\left|e_{\alpha}-f\right|\wedge f\right)_{\alpha\in A}$. Since $\bigwedge G=0_{F}$ there is $g\in G$ such that $f\not\le g$, i.e. $\left(f-g\right)^{+}>0_{F}$. Let $\alpha_{0}$ be such that $\left|f-e_{\alpha}\right|\wedge f\le g$, for every $\alpha\ge \alpha_{0}$. If $\alpha\ge \alpha_{0}$, then $e_{\alpha}\ge 0_{F}$ and $e_{\alpha}\ge f-\left|f-e_{\alpha}\right|$ yield $e_{\alpha}\ge \left(f-\left|f-e_{\alpha}\right|\right)^{+}=f-f\wedge\left|f-e_{\alpha}\right| \ge f-g$; using $e_{\alpha}\ge 0_{F}$ again leads to $e_{\alpha}\ge\left(f-g\right)^{+}>0_{F}$. Hence, $\bigwedge_{F}\left\{e_{\alpha}\right\}_{\alpha\ge\alpha_{0}}\ne 0_{F}$  (the infimum is either strictly positive or does not exist), and since $E$ is regular, there is $e\in E$ such that $e_{\alpha}\ge e>0_{F}$, for every $\alpha\ge\alpha_{0}$. As $e+F_{+}$ is uo-closed, we conclude that $f=\lim\limits_{\alpha\ge\alpha_{0}}e_{\alpha}\ge e$. Since $f$ was arbitrary, $E$ is order dense in $\overline{E}_{uo}^{1}$.\medskip

Consider the inclusion $J_{\overline{E}^{1}_{uo}}$ from $\overline{E}^{1}_{uo}$ into $F$. It is continuous with respect to the induced uo convergence on $\overline{E}^{1}_{uo}$ (NOT the uo convergence of $\overline{E}^{1}_{uo}$), and moreover, $E$ is dense and order dense in $\overline{E}^{1}_{uo}$. Since $\left.J_{\overline{E}^{1}_{uo}}\right|_{E}$ is the inclusion $J_{E}$ of $E$ into $F$, which is order continuous, it follows from Proposition \ref{oc} that $J_{\overline{E}^{1}_{uo}}$ is order continuous, and so $\overline{E}^{1}_{uo}$ is regular. Applying the established parts of the proof to $\overline{E}^{1}_{uo}$, we conclude that $\overline{E}^{1}_{uo}$ is order dense in $\overline{E}^{2}_{uo}=\overline{\overline{E}^{1}_{uo}}^{1}_{uo}$, which is a regular sublattice. Due to transitivity of order density, $E$ is order dense in $\overline{E}^{2}_{uo}$.\medskip

Now assume that $H$ is a regular sublattice of $F$ such that $E$ is order dense in $H$. Then, $h=\bigvee_{H}\left(\left[0_{F},h\right]\cap E\right)$, for every $h\in H_{+}$. Since $H$ is regular, the supremum can be taken in $F$, and so $\left[0_{F},h\right]\cap E$ is an increasing net in $E$ whose supremum (and so order limit) is $h$. Hence, $h$ belongs to the order adherence $\overline{E}_{o}^{1}$ of $E$, and so $H_{+}\subset \overline{E}_{o}^{1}$, which yields $H\subset \overline{E}_{o}^{1}$. In particular, since $\overline{E}^{2}_{uo}$ is regular, and $E$ is order dense in $\overline{E}^{2}_{uo}$, it follows that $\overline{E}_{o}^{1}\subset \overline{E}_{uo}^{1}\subset \overline{E}^{2}_{uo}\subset \overline{E}^{1}_{o}$, and so all these three sets are equal (and also equal to $\overline{E}_{o}^{2}$). Thus, $\overline{E}^{1}_{o}$ is uo-closed.
\end{proof}

\begin{remark}\label{reg4}It also follows from the proof that if $E$ is regular, then $\overline{E}_{o}^{1}$ is uo-closed, so that $\overline{E}_{o}^{1}= \overline{E}_{uo}^{1}=\overline{E}_{o}=\overline{E}_{uo}$. Note that $\left(\overline{E}_{o}\right)_{+}$ consists of the supremums of increasing nets in $E$. Hence, if $E$ contains all such supremums, then it is order closed.
\qed\end{remark}

\begin{corollary}\label{reg1}If $E$ is a regular sublattice of $F$, and $H$ is a sublattice such that $E\subset H\subset \overline{E}_{o}$, then $H$ is regular and $E$ is order dense in $H$.
\end{corollary}
\begin{proof}
If $E$ is regular, then it is order dense in $\overline{E}_{o}$, and the latter is regular. If $E\subset H\subset \overline{E}_{o}$, then $H$ is order dense in $\overline{E}_{o}$, hence regular, and so regular in $F$ due to transitivity of regularity.
\end{proof}

An order complete vector lattice $H$ is \emph{universally complete} if every set of mutually disjoint vectors is order bounded. Every Archimedean vector lattice $F$ embeds as an order dense sublattice into a unique universally complete vector lattice $F^{u}$, called the \emph{universal completion} of $F$ (see \cite[Theorem 7.23]{ab0}). The following was conjectured by Vladimir Troitsky.

\begin{corollary}
If $E$ is a regular sublattice of $F$, then $\overline{E}^{F^{u}}_{o}$ (the order closure of $E$ viewed as a sublattice of $F^{u}$) is the universal completion of $E$, and $\overline{E}^{F}_{o}=\overline{E}^{F^{u}}_{o}\cap F$.
\end{corollary}
\begin{proof}
Since $E$ is regular in $F$, and $F$ is order dense in $F^{u}$, it follows that $E$ is regular in $F^{u}$, and so $E$ is order dense in $H=\overline{E}^{F^{u}}_{o}$, which is regular in $F^{u}$, according to Theorem \ref{rod}. In order to prove that $H=E^{u}$ it is left to show that $H$ is universally complete. Assume that $G\subset H$ is order bounded in $F^{u}$. Then $G$ has a supremum $g\in F^{u}$ which is also a supremum of an increasing net $G^{\vee}\subset H$. Therefore, $g$ is an order limit of that net, and so it belongs to $H$ since the latter is order closed. It is easy to see that $g$ is the supremum of $G$ in $H$. Since any order bounded set in $H$ is also order bounded in $F^{u}$, as well as any set in $H$ that consists of disjoint vectors, we conclude that $H$ is universally complete. The last claim follows from the fact that according to Theorem \ref{gt1}, the inclusion of $F$ into $F^{u}$ is an uo-embedding.
\end{proof}

We will call a locally solid additive convergence $\eta$ of $F$ \emph{sturdy} if $\overline{E}^{1}_{\eta}\subset\overline{E}_{o}$, for every regular sublattice $E$. Then $\left\{0_{E}\right\}$ is closed, and so a sturdy convergence is always Hausdorff. A \emph{sturdy vector lattice} is a vector lattice endowed with a sturdy locally solid additive convergence. It follows from Theorem \ref{rod} and Corollary \ref{reg1} that $F$ is sturdy if and only if every regular sublattice is order dense in its adherence, which is itself a regular sublattice. In fact, the last assertion can be slightly improved.

\begin{corollary}\label{reg2}Every regular sublattice of a sturdy vector lattice $\left(F,\eta\right)$ is order dense in its closure, which is itself a regular sublattice.
\end{corollary}
\begin{proof}
Let $E$ be a regular sublattice of $F$. It follows from Theorem \ref{rod} and Remark \ref{reg4} that $\overline{E}_{o}$ is regular and order closed. Hence, $\overline{\overline{E}_{o}}^{1}_{\eta}\subset \overline{\overline{E}_{o}}_{o}=\overline{E}_{o}$, and so $\overline{E}_{o}$ is closed in $\left(F,\eta\right)$. Thus, $\overline{E}_{\eta}\subset \overline{E}_{o}$, and so $E$ is order dense in $\overline{E}_{\eta}$, which is itself a regular sublattice, according to Corollary \ref{reg1}.
\end{proof}

In particular, every topologically dense regular sublattice of a sturdy convergence vector lattice is order dense.

\begin{corollary}\label{doc}Let $F$ be a sturdy vector lattice, and let $H\subset F$ be a topologically dense regular sublattice. If $T:F\to E$ is a continuous operator into a vector lattice $E$ endowed with a Hausdorff locally solid additive convergence such that $\left.T\right|_{H}$ is positive and order continuous, then $T$ is positive and order continuous.
\end{corollary}
\begin{proof}
Since $E_{+}$ is closed and $T$ is continuous, $T^{-1}E_{+}$ is a closed set that contains $H_{+}$. Since $H$ is topologically dense, it follows that $F_{+}=\overline{H}_{+}=\overline{H_{+}}\subset T^{-1}E_{+}$, from where $TF_{+}\subset E_{+}$, and so $T$ is positive. According to the comment above, $H$ is order dense, and so $T$ is order continuous by Proposition \ref{oc}.
\end{proof}

\begin{example}
From Remark \ref{reg4} any locally solid convergence on $F$ which is stronger than uo convergence is sturdy. Hence, uo convergence itself, as well as order and uniform convergences are sturdy. Another example is the compact-open topology on an almost locally compact topological space, by virtue of Theorem \ref{couo2}. If $F$ is a Banach lattice, then while the norm convergence is not necessarily stronger than uo convergence (because e.g. norm convergence $L_{p}\left[0,1\right]$ does not imply a.e. convergence), the norm-closure coincides with the uniform adherence (this follows from \cite[Lemma 1.1]{tt}). Hence, the norm convergence on a Banach lattice is sturdy.
\qed\end{example}

We will now consider a non-sturdy locally solid topological vector lattice.

\begin{example}\label{nons}
Let $X$ be a topological space formed as a topological sum (disconnected union) of $\Q$ with $\R\backslash \Q$ (or more generally, a topological sum of two complementary dense subsets of some topological space). It follows from Stone-Weierstrass theorem that $\Co\left(\R\right)$ is dense in $\Co\left(X\right)$, which is endowed with the compact-open topology. We will show that $\Co\left(\R\right)$ is regular, but not order dense in $\Co\left(X\right)$, which violates sturdiness. This can be deduced from \cite[Example 6.1, Proposition 6.5 and Theorem 7.1]{erz}, but here is a more direct argument using \cite[Lemma 3.1]{et}. Since there is no $f\in \Co\left(\R\right)$ between $\0$ and $\1_{\Q}$, we conclude that $\Co\left(\R\right)$ is not order dense in $\Co\left(X\right)$. If $G\subset \Co\left(\R\right)$ is such that $\bigwedge_{\Co\left(\R\right)} G=\0$, then for every $\varepsilon>0$ and an open nonempty $U\subset X$ there is an open nonempty $V\subset \Int\overline{U}^{\R}$ and $g\in G$ such that $\left.g\right|_{V}\le\varepsilon$. Since $V\cap U$ is an open in $X$ nonempty subset of $U$, we conclude that $\bigwedge_{\Co\left(X\right)} G=\0$. Thus, $\Co\left(\R\right)$ is regular in $\Co\left(X\right)$.
\qed\end{example}

Modifying the idea of Example \ref{nons} we will construct a regular sublattice whose closure is not regular.

\begin{example}\label{nonr}
Let $X=\Q\times\N_{0}$, where $\N_{0}=\N\cup\left\{0\right\}$. Note that $X$ is homeomorphic to $\Q$. Let $\left\{q_{n},~n\in\N\right\}$ be an enumeration of the rational numbers. Define $\varphi:X\to\R$ by $\varphi\left(p,0\right)=p$ and $\varphi\left(q,n\right)=\pi q+q_{n}$, for $p,q\in\Q$ and $n\in\N$; note that $\varphi$ is injective on $\Q\times\N$ and on $\Q\times\left\{0\right\}$, and if $\varphi\left(p,0\right)=\varphi\left(q,n\right)$, then $p=q_{n}$ and $q=0$.

Using a similar tactic as in Example \ref{nons} one can show that the lattice $E=\left\{f\circ\varphi,~ f\in \Co\left(\R\right)\right\}$ is a regular sublattice of $\Co\left(X\right)$. Recall that by Stone-Weierstrass theorem the closure of a sublattice consists of all continuous functions, which satisfy the same constraints as all elements of the sublattice (a constraint is an equality of the type $\alpha f\left(x\right)=\beta f\left(y\right)$, for $\alpha,\beta\ge0$ and $x,y\in X$; the claim can be deduced from \cite[Theorem 16.5.5]{bn}, see also \cite[Theorem 2.1]{et2}). Hence, $\overline{E}$ consists of all functions $f\in\Co\left(X\right)$ such that $f\left(q_{n},0\right)=f\left(0,n\right)$, for every $n\in\N$. Let $f_{m}:X\to\R$ be defined as $f_{m}\left(p,0\right)=1$, and $f_{m}\left(q,n\right)=e^{-m\left|q\right|}$, for $p,q\in\Q$ and $m,n\in\N$. It is easy to see that $f_{m}\in \overline{E}$ and $f_{m}\ge \1_{\Q\times\left\{0\right\}}$, for every $m\in\N$, and so $f_{m}\not\downarrow_{\Co\left(X\right)}\0$. On the other hand, since $f_{m}\left(q,n\right)\downarrow0$, for every $q\ne 0$ and $n\in\N$, it follows that $f_{m}\downarrow_{\overline{E}}\0$. Thus, $\overline{E}$ is not regular.
\qed\end{example}

Let us finish the section with some questions regarding sturdiness.

\begin{question}
For which Tychonoff spaces $X$ is $\Co\left(X\right)$ sturdy?
\end{question}

As was mentioned before, almost local compactness of $X$ implies sturdiness of $\Co\left(X\right)$, but metrizability does not, as was demonstrated by Example \ref{nons}.

\begin{question}
What permanence properties does sturdiness have? For example, is a product of two sturdy convergence vector lattices sturdy?
\end{question}

It is easy to see that a regular sublattice of a sturdy convergence vector lattice is sturdy in the induced convergence.

\begin{question}
Is there a non-sturdy convergence vector lattice in which the adherence (or closure) of every regular sublattice is regular?
\end{question}

\begin{remark}
A possible candidate may arise from the following observation: if $F$ is order complete and the convergence $\eta$ is weaker than the order convergence, then every $\eta$-closed sublattice $E\subset F$ is regular. Indeed, if $\left(e_{\alpha}\right)_{\alpha\in A}$ is decreasing to $0_{F}$ in $E$, then it is order bounded from below, hence has an infimum $f\in F_{+}$ to which it order converges in $F$, thus converges in $\eta$, therefore $f\in E$, and so $f=\bigwedge_{E}\left(e_{\alpha}\right)_{\alpha\in A}=0_{F}$.
\qed\end{remark}

\section{Uniform convergence and $\sigma$-order continuity}\label{sigs}

In this section we will always consider Archimedean vector lattices endowed with the uniform convergence, as defined in Example \ref{ru}. An important property of the uniform convergence is that an operator is \emph{order bounded} (i.e. maps order intervals inside order intervals) if and only if it is continuous with respect to the uniform convergences (see \cite[Theorem 10.3]{dow}). Since every positive operator is clearly order bounded and uniform convergence is sturdy, we get the following corollaries of propositions \ref{loc0} and \ref{oc}, as well as Corollary \ref{doc}.

\begin{corollary}\label{loc}
In the notations of Proposition \ref{loc0}, $H$ is always uniformly closed.
\end{corollary}

\begin{corollary}\label{uoc}For a positive operator $T:F\to E$ between Archimedean vector lattices the following conditions are equivalent:
\item[(i)] $T$ is order continuous;
\item[(ii)] There is a topologically majorizing $G\subset F_{+}$ such that $\left.T\right|_{F_{g}}$ is order continuous, for every $g\in G$;
\item[(iii)] There is an order dense and topologically majorizing sublattice $H$ such that $\left.T\right|_{H}$ is order continuous.
\item[(iv)] There is a topologically dense regular sublattice $H$ such that $\left.T\right|_{H}$ is order continuous.
\end{corollary}

A positive operator $T:F\to E$ is $\sigma$\emph{-order continuous} if it maps sequences decreasing to $0_{F}$ into sequences decreasing to $0_{E}$. This property is weaker than the order continuity. If $T$ is a homomorphism, it is $\sigma$-order continuous iff $\bigwedge G=0_{F}$ implies $\bigwedge TG=0_{E}$, for any countable $G$. It is easy to see that Corollary \ref{loc} remains valid for $\sigma$-order continuity. The following result seems to be exclusively sequential and generalizes \cite[Theorem 6]{wnuk}.

\begin{theorem}\label{main1}
Let $T:F\to E$ be a homomorphism between Archimedean vector lattices. Assume that either $F$ or $E$ has $\sigma$-property. If $H$ is a sublattice of $F$ such that $\left.T\right|_{H}$ is $\sigma$-order continuous, then so are $\left.T\right|_{\overline{H}^{1}}$ and $\left.T\right|_{\overline{H}}$.
\end{theorem}
\begin{proof}
Let us describe the general strategy. Assume that $\left(h_{n}\right)_{n\in\N}$ is a sequence in $\overline{H}^{1}$ or $\overline{H}$, which decreases to $0_{F}$, but there is $e\in E$ such that $Th_{n}\ge e>0_{E}$, for every $n\in\N$. We will seek a contradiction, by constructing a countable $C\subset H_{+}$ such that $\left\{h_{n}\right\}_{n\in\N}\subset \overline{C}$ and $TC\ge g\notin -E_{+}$. Since $T$ is positive it will imply that $TC\ge g^{+}>0_{E}$. From $\sigma$-order continuity of $\left.T\right|_{H}$ it will follow that $\bigwedge C\ne 0_{F}$ (it is either strictly positive or does not exist), hence there $h\in F_{+}\backslash\left\{0_{F}\right\}$ such that $C\subset F_{+}+h$; as the latter set is closed, it follows that $\left\{h_{n}\right\}_{n\in\N}\subset \overline{C}\subset F_{+}+h$, which contradicts $\bigwedge \left\{h_{n}\right\}_{n\in\N}=0_{F}$.\medskip

We first consider the case when $F$ has $\sigma$-property, and so $\overline{H}^{1}=\overline{H}$. There are $\left\{f_{n}\right\}_{n\in\N}\subset F_{+}$ and $\left\{g_{mn}\right\}_{m,n\in\N}\subset H_{+}$ such that $h_{n}\le f_{n}$ and $\left|h_{n}-g_{mn}\right|\le \frac{1}{m}f_{n}$, for every $m,n\in\N$. Then, there is $f\in F_{+}$ such that $\left\{f_{n}\right\}_{n\in\N}\subset F_{f}$. Since $E$ is Archimedean, replacing $f$ with $\varepsilon f$, for some $\varepsilon>0$, if needed, we may assume that $Tf\not\ge e$. For every $n\in\N$ and $m\ge \|f_{n}\|_{f}$ we have $\left|h_{n}-g_{mn}\right|\le \frac{1}{m}f_{n}\le \frac{\|f_{n}\|_{f}}{m}f\le f$, and since $T\ge 0$, this implies $$Tg_{mn}\ge Th_{n}-\left|Th_{n}-Tg_{mn}\right|= Th_{n}-T\left|h_{n}-g_{mn}\right|\ge e-Tf\notin -E_{+}.$$ Hence $C=\left\{g_{mn}\right\}_{m\ge \|f_{n}\|_{f}}$ fulfills the role described in the first paragraph.\medskip

Now assume that $E$ has $\sigma$-property, and let us prove the claim for $\overline{H}^{1}$. Let  $e$, $\left\{f_{n}\right\}_{n\in\N}$ and $\left\{g_{mn}\right\}_{m,n\in\N}$ be as above. There is $g\in E$ such that $\left\{Tf_{n}\right\}_{n\in\N}\subset E_{g}$, and $g\not\ge e$. For every $n\in\N$ and $m\ge \|Tf_{n}\|_{g}$ we have $Tg_{mn}\ge Th_{n}-T\left|h_{n}-g_{mn}\right|\ge e-\frac{1}{m}Tf_{n}\ge e-g\notin -E_{+}$. Therefore, $C=\left\{g_{mn}\right\}_{m\ge \|Tf_{n}\|_{g}}$ fulfills the role described in the first paragraph. The proof for $\overline{H}$ is based on the same idea, but requires more technical work.\medskip

\textbf{Claim.} If $A,D\subset F_{+}$ are such that $A$ is countable and $A\subset\overline{D}$, then there is $u\in E_{+}$ and a countable $C\subset D$ such that $TA\cup TC\subset E_{u}$ and $A\subset\overline{C}$, within $\left\{f\in F,~ T\left|f\right|\in E_{u}\right\}$ (which is an ideal in $F$).\medskip

We will prove by induction that if in the claim $A\subset\overline{D}^{l}$, for a countable ordinal $l$, then $C$ can be chosen so that $A\subset\overline{C}^{l}$. Note that this will prove the claim, since according to \cite[Theorem 3.3]{lm} there is a countable ordinal $l$ such that $A\in \overline{D}^{l}$. If $l=1$, for every element of $a\in A$ there are $v_{a}\in F_{+}$ and a sequence $D_{a}\subset F_{v_{a}}\cap D$ that converges to $a$ in $\|\cdot\|_{v_{a}}$; since $A$ is countable, due to $\sigma$-property of $E$, there is $u\in E_{+}$ such that $E_{u}$ contains $TA$ as well as $Tv_{a}$, for every $a$; it is easy to see that $u$ and $C=\bigcup\limits_{a\in A}D_{a}$ satisfy the requirements of the claim.\medskip

If the claim is proven for $l-1$, there are $u',u''\in E_{+}$ and countable $C'\subset \overline{D}^{l-1}$ and $C\subset D$ such that $TA\cup TC'\subset E_{u'}$, $TC'\cup TC\subset E_{u''}$ and $A\subset\overline{C'}^{1}$ and $C'\subset\overline{C}^{l-1}$ within $\left\{f\in F,~ T\left|f\right|\in E_{u'}\right\}$ and $\left\{f\in F,~ T\left|f\right|\in E_{u''}\right\}$, respectively. Since the uniform adherence within a larger sublattice is larger than within a smaller one, we conclude that $C$ and $u=u'+u''$ satisfy the requirements of the claim.

If $l$ is a limit ordinal, then $\overline{D}^{l}=\bigcup\limits_{k<l}\overline{D}^{k}$, and so for every $a\in A$ there is $k_{a}<l$ such that $a\in \overline{D}^{k_{a}}$; hence, there are $u_{a}\in E_{+}$ and a countable $C_{a}\subset D$ such that $\left\{Ta\right\}\cup TC_{a}\subset E_{u_{a}}$ and $a\in\overline{C_{a}}^{k_{a}}$ within $\left\{f\in F,~ T\left|f\right|\in E_{u_{a}}\right\}$. Any $u$ such that $\left\{u_{a}\right\}_{a\in A}\subset E_{u}$ and $C=\bigcup\limits_{a\in A}C_{a}$ satisfy the requirements of the claim.\qed\medskip

Now assume that $\left\{h_{n}\right\}\subset \overline{H}$ is such that  $h_{n}\downarrow 0_{F}$, but there is $e\in E$ such that $Th_{n}\ge e>0_{E}$, for every $n\in\N$. According to the claim, by restricting $T$ we may assume that $E=E_{u}$, for some $u\in E_{+}$. Then, the uniform convergence on $E$ is equivalent to the convergence with respect to $\|\cdot\|_{u}$. As $T$ is positive, it is still uniform-to-uniform continuous, and so continuous with respect to the topologies generated by the uniform convergences (which in the case of $E$ is just the $\|\cdot\|_{u}$-topology).\medskip

By scaling $u$ we may assume that $u\not\ge e$, and so $\left(e-u\right)^{+}>0_{E}$.  Let $B$ be the open unit ball in $E$ with respect to $\|\cdot\|_{u}$. There is a countable $C\subset H_{+}$ such that $\left\{h_{n}\right\}\subset \overline{C}$. For every $n$ we have that $h_{n}+T^{-1}B$ is a neighborhood of $h_{n}$, and so $h_{n}\in \overline{C\cap \left(h_{n}+T^{-1}B\right)}$. Using the claim again, find a countable $D\subset C\cap \left(\left\{h_{n}\right\}_{n\in\N}+T^{-1}B\right)$ such that $\left\{h_{n}\right\}_{n\in\N}\subset\overline{D}$. For every $f\in D$ there is $n\in\N$ such that $f\in h_{n}+T^{-1}B$, hence $Tf-Th_{n}\in \left[-u,u\right]$, from where $Tf\ge Th_{n}-u\ge e-u\notin -E_{+}$. Therefore, $D$ fulfills the role described in the first paragraph.\end{proof}

\begin{remark}
In the light of the similarity of the proofs of the theorem and Theorem \ref{rod}, it is natural to expect these two results to have a common generalization, which is yet to be discovered.
\qed\end{remark}

Recall that $F$ has \emph{countable supremum property} if for every $G\subset F$ and $g\in F$ such that $g=\bigvee G$ there is a sequence $\left(g_{n}\right)_{n\in\N}\subset G$ such that $\bigvee\limits_{n\in\N}g_{n}=g$. One can show that in this case every $\sigma$-order continuous operator from $F$ is order continuous, which yields the following corollary.

\begin{corollary}
Let $T:F\to E$ be a homomorphism between Archimedean vector lattices. If $F$ has $\sigma$-property and countable supremum property and $H$ is a sublattice of $F$ such that $\left.T\right|_{H}$ is order continuous, then so is $\left.T\right|_{\overline{H}}$.
\end{corollary}

\section{Few remarks on $\sigma$-order convergence}\label{ssig}

We will conclude this article with discussing a version of Corollary \ref{uoc} for $\sigma$-order continuity. Recall that a sublattice $E$ of $F$ is $\sigma$\emph{-regular}, if for every countable $G\subset E$ such that $\bigwedge_{E}G=0_{F}$ we have $\bigwedge_{F}G=0_{F}$. A sequence $\left(f_{n}\right)_{n\in\N}\subset F$ is said to $\sigma$\emph{-order converge} to $f\in F$ if there is a decreasing $\left(g_{n}\right)_{n\in\N}\subset F_{+}$ such that $g_{n}\downarrow 0_{F}$ and $\left|f-f_{n}\right|\le g_{n}$, for every $n\in\N$. Note that $\sigma$-order convergence fits into the framework of the \emph{sequential convergence structures} (see \cite[Section 1.7]{bb}). The following is proven similarly to Theorem \ref{ordex} and a part of Theorem \ref{rod}.

\begin{proposition}
\item[(i)] $\sigma$-order convergence is the strongest locally solid additive sequential convergence structure in which every sequence which decreases to $0_{F}$ converges to $0_{F}$.
\item[(ii)] If $E$ is a $\sigma$-regular sublattice of $F$, then it is order dense in its $\sigma$-order adherence.
\end{proposition}

Recall that a sublattice $E$ of $F$ is \emph{super order dense} if for every $f\in F_{+}$ there is a countable $G\subset E$ such that $f=\bigvee G$. It is clear that this condition is stronger than order density and a super order dense sublattice is dense with respect to $\sigma$-order convergence. It is natural to ask whether Theorem \ref{rod} admits a ``full $\sigma$-version''.

\begin{question}
Is it true that a $\sigma$-regular sublattice is super order dense in its $\sigma$-order adherence? Is this adherence $\sigma$-order closed and $\sigma$-regular?
\end{question}

It is easy to see that this is true for ideals. Let us present a somewhat related result. Note that a sublattice $E\subset F$ is majorizing if and only if for every $f\in F_{+}$ there is $e\in E$ such that $f\le e$. In the following proposition the adherence and closure are taken with respect to the uniform convergence.

\begin{proposition}\label{msod}
If $E$ is a majorizing sublattice of $F$, then $\overline{E}^{1}_{+}$ consists of the limits of increasing sequences. Moreover, $E$ is super order dense in $\overline{E}$.
\end{proposition}
\begin{proof}
Let $f\in \overline{E}^{1}_{+}$, so that there is $g\in F_{+}$ and $\left(e_{n}\right)_{n\in\N}\subset E_{+}$ such that $\left|f-e_{n}\right|\le\frac{1}{n}g$, for every $n\in\N$. Since $E$ is majorizing, there is $e\in E$ such that $e\ge g$, and so $\left|f-e_{n}\right|\le\frac{1}{n}e$, for every $n\in\N$; we have $e_{n}\in \left[f-\frac{1}{n}e, f+\frac{1}{n}e\right]$. For $n\in\N$ define $f_{n}=\left(e_{n}-\frac{1}{n}e\right)^{+}\in E\cap \left[0,f\right]$; we have $\left|f-f_{n}\right|\le\frac{2}{n}e$. Finally, taking $h_{n}=\bigvee\limits_{k=1}^{n}f_{n}$, for $n\in\N$, produces an increasing sequence $\left(h_{n}\right)_{n\in\N}\subset E\cap \left[0_{F},f\right]$ that uniformly converges to $f$. In particular, $f=\bigvee\limits_{n\in\N}h_{n}$.\medskip

Let us prove by induction that $E$ is super order dense in $\overline{E}^{n}$, for every ordinal $n$. If the claim is proven for $n-1$, then $E$ is super order dense in $\overline{E}^{n-1}$, which in turn is majorizing in $F$, and so super order dense in $\overline{E}^{n}$. It is easy to see that super order density is transitive, and so the claim is proven for $n$. If $n$ is a limit ordinal, and $f\in \overline{E}^{n}_{+}$, then $f\in \overline{E}^{m}_{+}$, for some $m<n$, and since $E$ is super order dense in $\overline{E}^{m}$, there is a countable $G\subset E$ such that $f=\bigvee G$. As $f$ was arbitrary, we conclude that $E$ is super order dense in $\overline{E}^{n}$. Since the closure is an adherence of a sufficiently large degree the second claim is proven.
\end{proof}

An operator is $\sigma$-order continuous if it preserves $\sigma$-order limits. It is easy to see that for positive operators this definition is consistent with the one given at the beginning of Section \ref{sigs}. Note that the proof in Theorem \ref{uoc} works if $\sigma$-order continuity is understood as the continuity with respect to the restriction of the $\sigma$-order convergence on $F$ to $H$ and $\overline{H}$.\medskip

The following is proven similarly to Proposition \ref{oc} (including re-proving a part of \cite[Proposition 3.1]{erz}).

\begin{proposition}For a positive continuous operator $T:F\to E$ between Archimedean vector lattices endowed with locally solid additive convergences TFAE:
\item[(i)] $T$ is $\sigma$-order continuous;
\item[(ii)] There is a topologically majorizing $G\subset F_{+}$ such that $\left.T\right|_{F_{g}}$ is $\sigma$-order continuous, for every $g\in G$;
\item[(iii)] There is a topologically dense ideal $H$ such that $\left.T\right|_{H}$ is $\sigma$-order continuous.
\item[(iv)] There is a super order dense and topologically majorizing sublattice $H$ such that $\left.T\right|_{H}$ is $\sigma$-order continuous.
\end{proposition}

To reconstruct the ``full $\sigma$-version'' of Corollary \ref{uoc} we need to show that a topologically dense (with respect to the uniform convergence) $\sigma$-regular sublattice is super order dense. Note that we cannot combine Proposition \ref{msod} with the comment that precedes it, since a topologically dense sublattice is not necessarily topologically dense in the ideal that it generates.

\section{Acknowledgements}

The author wants to thank Vladimir Troitsky, Mark Roelands and Marten Wortel for valuable discussions on the topic of this paper, as well as Taras Banakh who contributed an idea for Example \ref{nonr} and the service \href{mathoverflow.com/}{MathOverflow} which made it possible. Additional credit goes to the reviewer, who has corrected a serious mistake, numerous minor errors, and whose comments increased readability of the paper.

\end{document}